\documentclass[fontsize=12pt,a4paper,DIV=10]{scrartcl}
\usepackage[utf8]{inputenc}
\usepackage[T1]{fontenc}
\usepackage[fleqn]{amsmath}
\usepackage{amsthm,amstext,amssymb,amsfonts,mathptmx}
\usepackage[babel,english=british]{csquotes}
\usepackage[english]{babel}
\usepackage[numbers]{natbib}
\usepackage{booktabs,calc,scrpage2,graphicx,subfigure,bm}
\usepackage{listings,xcolor,enumerate}
\usepackage{caption}
\usepackage{hyperref}

\setkomafont{disposition}{\normalfont\bfseries}
\newcommand{\vect}[1]{\ensuremath{{\bm{\mathbf{#1}}}}}
\newcommand{\mat}[1]{\ensuremath{\mathbf{#1}}}
\newcommand{\lattice}[1]{\ensuremath{\Lambda(#1)}}
\newcommand{\patSet}[1]{\ensuremath{\mathcal{P}(#1)}}
\newcommand{\patIndex}[1]{\ensuremath{\mathbb E_{#1}}}
\newcommand{\gGroup}[1]{\ensuremath{\mathcal{G}(#1)}}
\newcommand{\calcMod}[2]{\ensuremath\left.#1\right|_{#2}}
\newcommand{\euler}{\ensuremath{\mathrm{e}}}
\newcommand{\imag}{\ensuremath{\mathrm{i}}}
\newtheorem{thm}{Theorem}
\newtheorem{lem}{Lemma}

\newtheorem{rem}{Remark}
\newtheorem{ex}{Example}
\title{\vspace{-1\baselineskip}\huge The fast Fourier Transform\\ and fast Wavelet Transform\\ for Patterns on the Torus}
\author{Ronny Bergmann\thanks{Institute of Mathematics, University of Lübeck, Ratzeburger Allee 160, 23562 Lübeck, Germany.\newline bergmann@math.uni-luebeck.de}}
\date{June 18, 2012}
\newcommand{\theabstract}{We introduce a fast Fourier transform on regular $d$-dimensional lattices. We investigate properties of congruence class representants, i.e. their ordering, to classify directions and derive a Cooley-Tukey-Algorithm.
Despite the fast Fourier techniques itself, there is also the advantage of this transform to be parallelized efficiently, yielding faster versions than the one-dimensional Fourier transform. These properties of the lattice can further be used to perform a fast multivariate wavelet decomposition, where the wavelets are given as trigonometric polynomials. Furthermore the preferred directions of the decomposition itself can be characterised.
}
\hypersetup{pdfauthor=Ronny Bergmann,
          pdftitle={The fast Fourier Transform and fast Wavelet Transform for Patterns on the Torus},
          pdfsubject=\theabstract,
          plainpages=false,
          pdfstartview=FitH,       %Starteinstellung für Zoom: Seitenbreite
          pdfview=FitH,            %Zoom auf Seitenbreite nach jedem Link
          pdfpagemode=UseOutlines, %Kapitelüberschriften als Bookmarks
          bookmarksnumbered=true,  %PDF-Inhaltsverzeichnis nummerieren
          bookmarksopen=false,      %Inhaltsverzeichnis expandiert starten
          bookmarksopenlevel=0,    %nur bis einschließlich subsections expandieren
          colorlinks=true,         %Links farbig (nicht umrahmt) darstellen
          linkcolor=blue,
          citecolor=blue!75!black,
          urlcolor=black}
\begin{document}\thispagestyle{empty}
	\maketitle
	\begin{abstract}
		\theabstract
	\end{abstract}
	\paragraph{Keywords.} wavelets, lattices, multivariate fast Fourier transform, periodic multiresolution analysis, Dirichlet wavelets, shift invariant space 
	%
	%
	% was: input introduction.tex
	%
	% Introduction
	% 
	%    Ronny Bergmann
	% Created 2011-01-17
	%
	\section{Introduction} % (fold)
	\label{sec:introduction}

	Recently, a framework for multivariate periodic wavelet analysis~\citep{LangemannPrestin2010} was developed, using shift invariant spaces of translates defined by a regular integral matrix $\mat{M}$. We further investigate this periodic multiscale analysis and develop fast algorithms for the decomposition of a periodic multivariate function. This generalises the one-dimensional case of periodic wavelets as described e.g. in~\citep{Narcowich1996,PlonkaTasche1995,Se98}.

	There are many approaches towards decompositions of multivariate functions in terms of wavelets that have been investigated in the last two decades, e.g. curvelets~\citep{Candes2001,CaDo04}, ridgelets~\citep{Candes2003}, contourlets~\citep{DoVetterli2005} or shearlets~\citep{Dahlke2009}. The periodic wavelet transform is defined on a finite set of translates and hence---in contrast to the forementioned wavelets---leads to finite sums by construction. These translates are defined by the pattern of a matrix $\mat{M}$ and lead to a multivariate Fourier transform onto a frequency domain with the shape of a parallelogram. The corresponding multivariate Fourier matrix is a Kronecker product of one-dimensional Fourier matrices, as mentioned by~\citep{Auslander1995}.

	The main challenge for an implementation is the ordering of the elements. While the elements on the real line are naturally ordered, there are many different ways to order the points of a lattice. We combine the idea of using the Smith normal form, as described by Mersereau et al. in~\citep{MeBrGu83,MeSp81} with the theoretical results on abelian groups developed by Auslander et al. in~\citep{Auslander1995,AuJoJo96}. While the former authors use arbitrary decompositions of the regular matrix $\mat{M}$ to derive a multivariate Cooley-Tukey-Algorithm, we use the approach from the latter authors and split the pattern into subgroups. Furthermore, our investigation also improves the algorithm mentioned in~\citep{MeBrGu83} by using an order in terms of specific bases.

	Despite the mathematical motivation as a generalisation of the rectangular multivariate Fourier transform, the lattice based Fourier transform naturally arises in crystallography. Based on X-ray defractions, the shape of the so called unit cell and its internal structure of a crystal is examined, e.g. for proteins~\citep{DrenthMesters2007}. The reciprocal lattice used for the measurements of X-Ray defractions corresponds to the generating group from Section~\ref{sec:preliminaries}. The reciprocal lattice is derived by looking at biorthognal bases~\citep[Chapter 2]{Giacovazzo1992}, which are also used here in Section~\ref{sec:bases_for_the_pattern_and_the_generating_group}.

	The paper starts with some preliminary investigations and notations in Section~\ref{sec:preliminaries}. In Section~\ref{sec:bases_for_the_pattern_and_the_generating_group}, we define and characterise bases for the patterns and use these to address all elements. This also leads to the notion of the dimension of the pattern and describes an ordering with respect to specific bases. Using these orderings, there is no rearrangement necessary for the Fourier transform, neither in time nor in frequency. This is used in Section~\ref{sec:the_fast_fourier_transform_on_matm} to develop a fast algorithm for the Fourier transform on the pattern that can additionally be parallelized. In Section~\ref{sec:fast_wavelet_transform}, we describe properties of bases, especially together with the bases of the dual group---the so called generating group of a matrix $\mat{M}$. These are applied to the multivariate periodic wavelet transform developed in~\citep{LangemannPrestin2010} to obtain a fast algorithm for the decomposition. Furthermore, we are able to use the basis of the pattern and generating group to generalise the one-dimensional scaling property given in~\citep{Se98}, including the characterisation of directions and their transformation from one scaling space to the next. 

	Finally in Section~\ref{sec:example}, we discuss an example for both the Fourier transform and the wavelet transform. The first example is concentrated on computational costs of the multivariate Fourier transform and possible parallelizations. The wavelet transform is presented by decomposing two different box splines.
	% section introduction (end)
	% end of introduction.tex
	%
	%
	%
	%
	%
	% was: \input{preliminaries}
	\section{Preliminary investigations} % (fold)
	\label{sec:preliminaries}
	\subsection{Function spaces} % (fold)
	\label{sub:function_space}
	The space of functions under consideration is the Hilbert space $L^2(\mathbb T^d)$ of all square integrable functions on the torus $\mathbb T^d \cong [0,2\pi)^d$ with the inner product
	\begin{equation}\label{eq:inner-product}
	    \langle f,g \rangle 
	  = \frac{1}{(2\pi)^d}\int_{\mathbb T^d}f(\vect{x})\overline{g(\vect{x})}\,d\vect{x},\quad \text{ for } f,g \in L^2(\mathbb T^d)
	%	\text{,}
		\text{.}
	\end{equation}
	%such that $L^2(\mathbb T^d) := \left\{f:\,||f||^2 = \langle f,f \rangle < \infty \right\}$.
	Every function $f \in L^2(\mathbb T^d)$ can be decomposed with respect to the monomials $\euler^{\imag\vect{k}^T\circ}\in L^2(\mathbb T^d)$, $\vect{k}\in\mathbb Z^d$, which yields the Fourier series representation
	\begin{equation}\label{eq:fourier-series}
		f(\vect{x}) = \sum_{\vect{k} \in \mathbb Z^d} c_{\vect{k}}(f)\euler^{\imag\vect{k}^T\vect{x}},\quad\text{where}\quad c_{\vect{k}}(f) = \langle f,\euler^{\imag\vect{k}^T\circ}\rangle,\quad\vect{k\in\mathbb Z^d}\text{.}
	\end{equation}
	We denote by $\vect{c}(f) = \left(c_{\vect{k}}(f)\right)_{\vect{k}\in\mathbb Z^d} \in l^2(\mathbb Z^d)$ the generalised sequences which form a Hilbert space with the inner product
	\begin{equation*}%\label{eq:inner-product-series}
		\langle\vect{c},\vect{d}\rangle = \sum_{\vect{k} \in \mathbb Z^d}c_{\vect{k}}\overline{d_{\vect{k}}},\quad \vect{c},\vect{d}\in l^2(\mathbb Z^d)\text{,}
	\end{equation*}
	where the Parseval equation reads as
	\begin{equation*}
			\langle f, g \rangle = \langle \vect{c}(f),\vect{c}(g) \rangle = \sum_{\vect{k} \in \mathbb Z^d} c_{\vect{k}}(f) \overline{c_{\vect{k}}(g)}\text{.}
	\end{equation*}
	We define for any $\vect{y}\in \mathbb R^d$ the translation operator $T(\vect{y})f = f(\circ-2\pi\vect{y}),$ $f\in L^2(\mathbb T^d)$. A straight forward computation leads to $c_{\vect{k}}(T(\vect{y})f) = \euler^{-2\pi \imag \vect{k}^T \vect{y}}c_{\vect{k}}(f)$. We can restrict $\vect{y}$ to any shifted unit cube, e.g. $[-\tfrac{1}{2},\tfrac{1}{2})^d$, because $\euler^{-2\pi\imag \vect{k}^T\vect{z}} = 1,\quad \vect{k},\vect{z}\in\mathbb Z^d$.
	% subsubsection function_space (end)
	%
	%
	%
	%     Pattern & the generating group
	%
	%
	\subsection{The pattern and the generating group} % (fold)
	\label{sub:pattern_and_the_generating_group}
	For any regular matrix $\mat{M} \in \mathbb Z^{d\times d}$, we define the congruence relation for $\vect{h},\vect{k} \in \mathbb Z^d$ with respect to $\mat{M}$ by
	\begin{equation*}
			\vect{h} \equiv \vect{k} \bmod \mat{M} \Leftrightarrow \exists \vect{z} \in \mathbb Z^d: \vect{k} = \vect{h} + \mat{M}\vect{z}\text{.}
	\end{equation*}
	We define the %\emph{
	lattice
	%}
	\begin{equation*}
		\lattice{\mat{M}} := \mat{M}^{-1}\mathbb Z^d = \{\vect{y}\in\mathbb R^d : \mat{M}\vect{y} \in \mathbb Z^d\}\text{, }	
	\end{equation*}
	note that it is $1$-periodic and define the %\emph{
	pattern %}
	$\patSet{\mat{M}}$ 
	%of the regular integral matrix $\mat{M}$
	as any complete set of congruence class representants, e.g. $ \lattice{\mat{M}}\cap[0,1)^d$ or $ \lattice{\mat{M}}\cap\left[-\tfrac{1}{2},\tfrac{1}{2}\right)^d$, which both contain exactly one element of each congruence class with respect to $\bmod\ \mat{I}$ on $\lattice{\mat{M}}$.
	We denote by $[\vect{x}]_{\mat{M}}$ the congruence class of $\vect{x}\in\lattice{\mat{M}}$ and define by 
	\begin{equation*}
	\calcMod{\vect{x}}{\patSet{\mat{M}}}%
	:= [\vect{x}]_{\mat{M}}\cap\patSet{\mat{M}},\quad \vect{x}\in\lattice{\mat{M}}
	\end{equation*}
	the mapping from any point $\vect{x}\in\lattice{\mat{M}}$ of the lattice onto its congruence class representant belonging to $\patSet{\mat{M}}$. Then $(\patSet{\mat{M}}, \calcMod{+}{\patSet{\mat{M}}})$ is an abelian group. 
	%Here, the calculation $\bmod\mat{I}$ maps any $\vect{x}\in\lattice{\mat{M}}$ onto its congruence classs representant in $\patSet{\mat{M}}$. 

	Since  $\mat{M} \in\mathbb Z^{d\times d}$ is regular, it can be seen as a bijective map. Hence all the definitions and properties above also hold for the %\emph{
	generating group
	%}
	$\gGroup{\mat{M}} := \mat{M}\patSet{\mat{M}}$ equipped with $\calcMod{+}{\gGroup{\mat{M}}}$, where $\mat{M}\circ: \patSet{\mat{M}} \to \gGroup{\mat{M}}$ performs a group isomorphism between $(\gGroup{\mat{M}},\calcMod{+}{\gGroup{\mat{M}}})$ and $(\patSet{\mat{M}},\calcMod{+}{\patSet{\mat{M}}})$.

	The Smith normal form is defined by the decomposition
	\begin{equation}\label{eq:SNF}
		\mat{M} = \mat{Q}\mat{E}\mat{R},\quad \mat{Q},\mat{E},\mat{R}\in\mathbb Z^{d\times d},\text{ where } \mat{E} = \operatorname{diag}\left(\epsilon_j\right)_{j=1}^d,
	\end{equation}
	with $|\det \mat{R}| = |\det\mat{Q}| = 1$ and the elementary divisors $\epsilon_j \in \mathbb N$---which exist due to the Theorem on elementary divisors, see e.g. \cite[Chapter 10]{Ko72}---fulfill $\epsilon_j | \epsilon_{j+1},$ $j=1,\ldots,d-1$. Besides the already mentioned isomorphism this also implies that
	\begin{equation}\label{eq:group-isomorphisms}
		\gGroup{\mat{M}} \cong \gGroup{\mat{E}} \cong \patSet{\mat{E}} \cong \patSet{\mat{M}} \cong \mathcal C_{\epsilon_1} \otimes \mathcal C_{\epsilon_2} \otimes \cdots \otimes \mathcal C_{\epsilon_d}\text{,}
	\end{equation}
	where $\mathcal C_{z} = \{0,1,\ldots,z-1\}, z\in\mathbb Z$ denotes the cyclic group $\mathbb Z / z\mathbb Z$. Hence  $|\gGroup{\mat{M}}| = |\patSet{\mat{M}}| = \epsilon_1\cdot\ldots\cdot\epsilon_d = |\det{\mat{M}}| =:m$. We further denote by $d_{\mat{M}} := \#\{\epsilon_j > 1\}$ the number of cycles greater than $1$.

	For any decomposition of a regular matrix $\mat{M}=\mat{J}\mat{N}$, $\mat{N},\mat{J}\in\mathbb Z^{d\times d}$ it holds~\citep{LangemannPrestin2010}, that there exists a unique decomposition for $\vect{y}\in\patSet{\mat{M}}$
	\begin{equation}\label{eq:decomposition}
		\vect{y} = \calcMod{%
		(\vect{x} + \mat{N}^{-1}\vect{z})}{\patSet{\mat{M}}}, \quad\vect{x}\in\patSet{\mat{N}},\ \vect{z}\in\patSet{\mat{J}}
	\end{equation}
	which can also be applied to $\gGroup{\mat{M}^T}$ yielding for $\vect{h}\in\gGroup{\mat{M}^T}$ the unique decomposition 
	\begin{equation*}
		\vect{h}= \calcMod{(\vect{g}+\mat{N}^T\vect{k})}{\gGroup{\mat{M}^T}},\quad\vect{k}\in\gGroup{\mat{J}^T},\ \vect{g}\in\gGroup{\mat{N}^T}\text{.}
	\end{equation*}
	The Fourier transform on the pattern $\patSet{\mat{M}}$ is defined~\citep{ChuiChun:1994} by
		\begin{equation}\label{eq:Fouriermatrix}
			\begin{split}
				\mathcal F(\mat{M}) &= \frac{1}{\sqrt{m}}\left(\euler^{- 2\pi \imag \vect{h}^T\mat{M}^{-1}\vect{g}}\right)_{\vect{h} \in \gGroup{\mat{M}^T},\,\vect{g} \in \gGroup{\mat{M}}}
			%\\
			%&
			= 
			\frac{1}{\sqrt{m}}\left(\euler^{- 2\pi \imag \vect{h}^T\vect{y}}\right)_{\vect{h} \in \gGroup{\mat{M}^T},\, \vect{y} \in \patSet{\mat{M}}},	
			\end{split}
		\end{equation}
	where $\vect{h}\in\gGroup{\mat{M}^T}$ indicates the rows, $\vect{y} \in \patSet{\mat{M}}$ indicates the columns and the equality holds if the ordering of the elements in $\gGroup{\mat{M}}$ and $\patSet{\mat{M}}$ are identical with respect to the bijection $\mat{M}\circ$. The % \emph{
	discrete Fourier transform %} 
	on $\patSet{\mat{M}}$ is defined for a vector $\vect{a} = (a_{\vect{y}})_{\vect{y}\in\patSet{\mat{M}}}\in\mathbb C^m$ arranged in the same ordering as the columns in~\eqref{eq:Fouriermatrix} by
	\begin{equation*}
		\vect{\hat a} = (\hat a_{\vect{h}})_{\vect{h}\in\gGroup{\mat{M}^T}} = \mathcal F(\mat{M})\vect{a},
	\end{equation*}
	where the vector $\vect{\hat a}$ is arranged as the columns of $\mathcal F(\mat{M})$ in~\eqref{eq:Fouriermatrix}.
		%extension due to reviewer #3
		%
		%
		% was:	\input{lattice-example}
		\begin{ex}\label{ex:lattice1}
			We a look at the sheared and rotated 2-dimensional lattice generated by the matrix
			\begin{equation}\label{eq:lattice-example}
				\mat{M} = \begin{pmatrix}
					4 &-3\\ 4&5 
				\end{pmatrix}
				=
				\begin{pmatrix}
					1 &-1\\ 1&1 
				\end{pmatrix}
				\begin{pmatrix}
					4 & 1\\ 0&4 
				\end{pmatrix}\text{.}
			\end{equation}
			One way of choosing the pattern $\patSet{\mat{M}}$ is illustrated in Fig.~\ref{fig:Pattern-Example} on the left.
			\begin{figure}
				\centering
				\includegraphics[width=.47\textwidth]{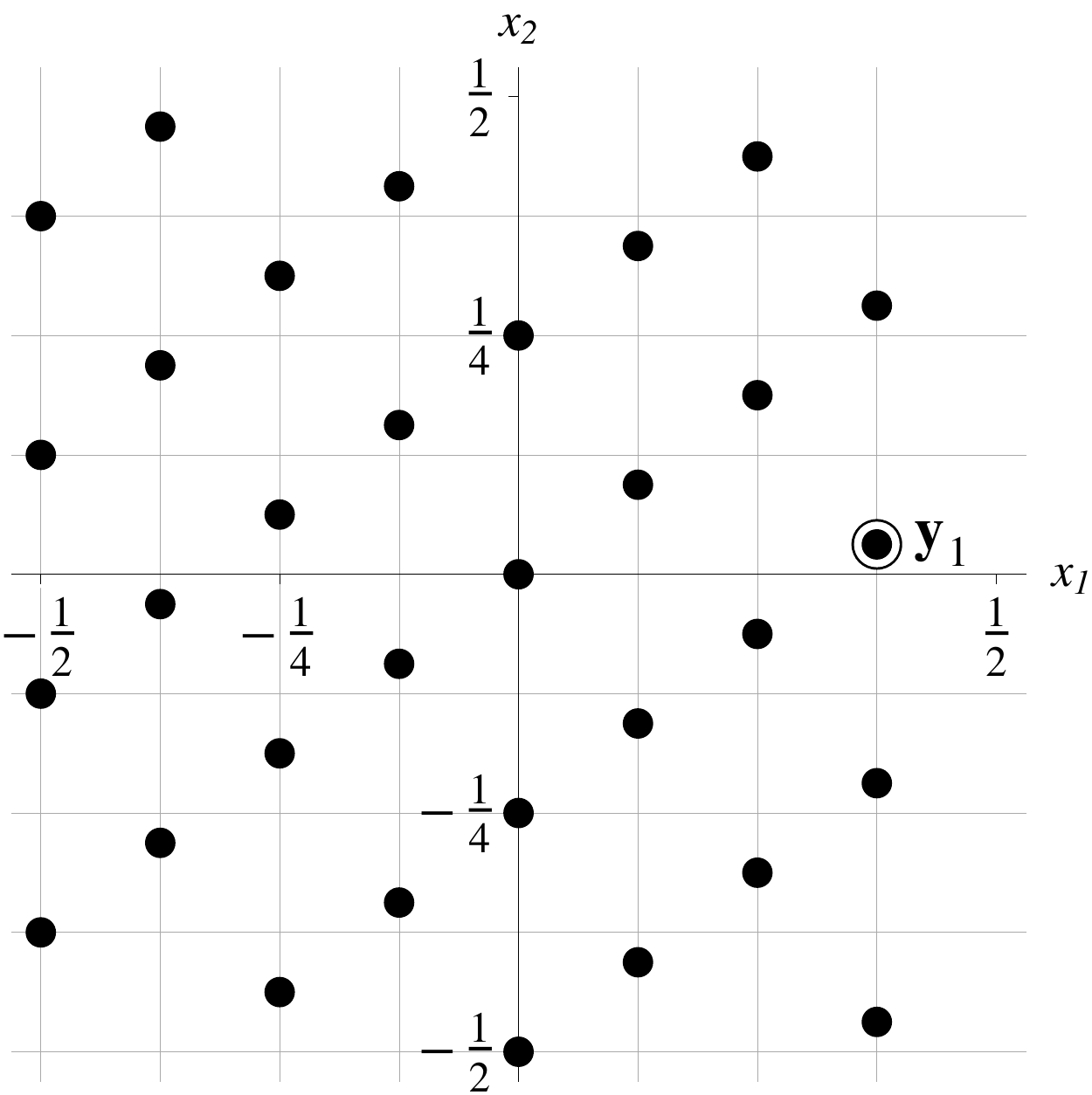}\hspace{.03\textwidth}
				\includegraphics[width=.47\textwidth]{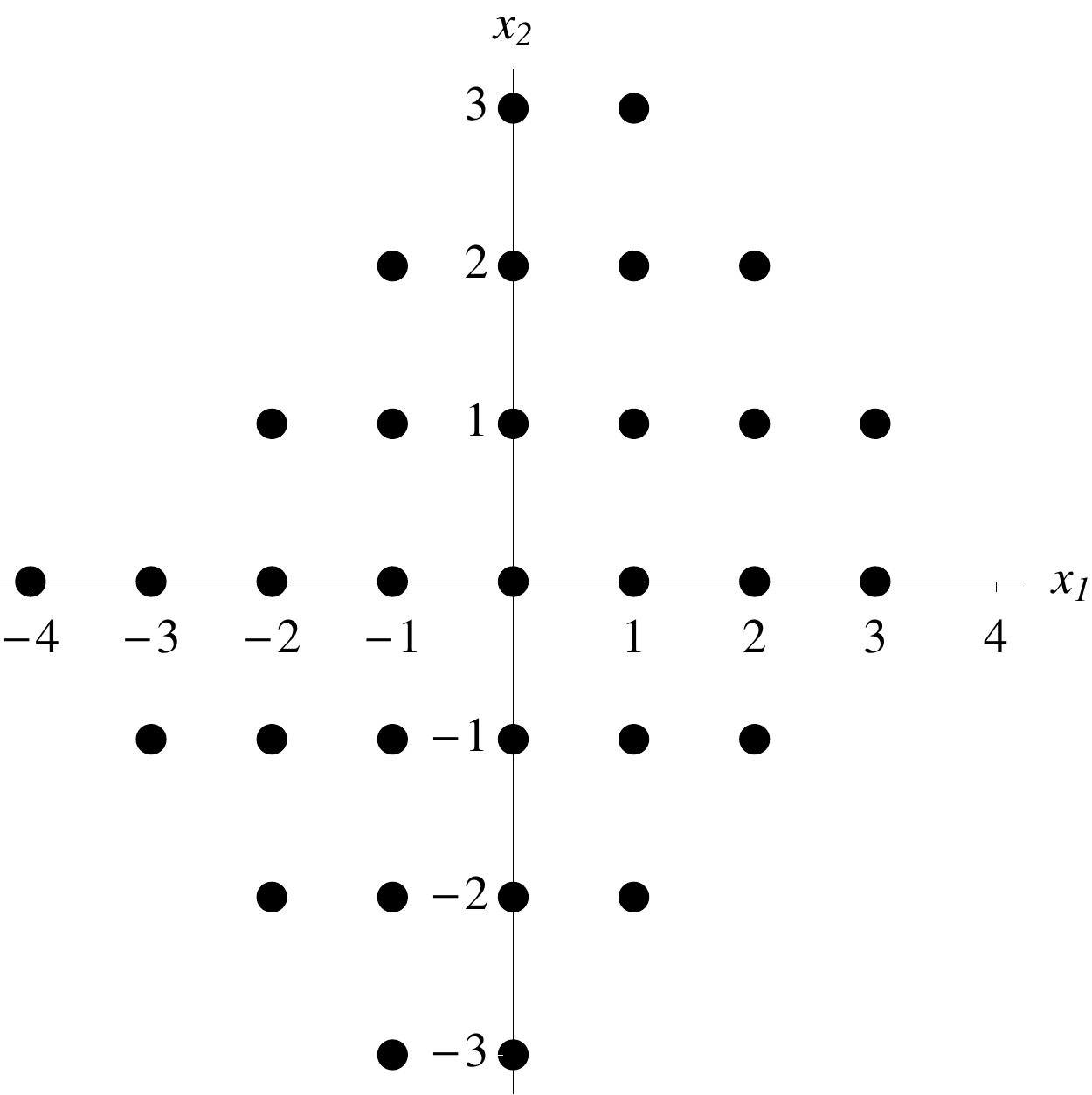}
				\caption{The pattern $\patSet{\mat{M}}$ for $\mat{M}$ given in~\eqref{eq:lattice-example} is illustrated on the left, where the set of congruence classes is chosen from \([-\tfrac{1}{2},\tfrac{1}{2})^2\). The emphazised point $\vect{y}_1 = \bigl(\tfrac{3}{8},\tfrac{1}{32}\bigr)^T$ is the canoncial basis vector of the pattern from Section~\ref{sec:bases_for_the_pattern_and_the_generating_group}. Its corresponding generating group $\gGroup{\mat{M}^T}$ (right) illustrates the set of frequencies obtained when performing a Fourier transform on the points of $\patSet{\mat{M}}$ using the matrix from Eq.~\eqref{eq:Fouriermatrix}.}
				\label{fig:Pattern-Example}
			\end{figure}
			The congruence class representants are chosen from $[-\tfrac{1}{2},\tfrac{1}{2})^2$. They can also be seen as sampling points on the 1-periodic torus. The lattice $\lattice{\mat{M}}$ consists of all integer shifts of these points. The corresponding generating group $\gGroup{\mat{M}^T}$, used to construct the Fourier matrix \eqref{eq:Fouriermatrix}, is given in Fig.~\ref{fig:Pattern-Example} on the right. They are obtained by taking all integer valued vectors from $\mat{M}^T[-\tfrac{1}{2},\tfrac{1}{2})^d$. The Smith normal form is given as 
			\begin{equation*}
				\mat{M} = \begin{pmatrix}
					4 &-3\\ 4&5 
				\end{pmatrix}
				= \begin{pmatrix}3&1\\5&2\end{pmatrix}
				\begin{pmatrix}1&0\\0&32\end{pmatrix}
				\begin{pmatrix}1&-12\\0&1\end{pmatrix}=\mat{Q}\mat{E}\mat{R}\text{,}
			\end{equation*}
			hence both groups are isomorphic to $\mathcal C_{32}$.
		\end{ex}
		% end: input lattice-example
		%
		%
		% section preliminaries (end)	
	% end: input{preliminaries}
	%
	%
	%
	%
	% was: \input{bases}
	\section{Bases for the pattern and the generating group} % (fold)
	\label{sec:bases_for_the_pattern_and_the_generating_group}
	In contrast to the one-dimensional pattern, i.e. the set $\{0,1/N,\ldots,(N-1)/N\}$, the ordering of the elements in $\patSet{\mat{M}}$ and $\gGroup{\mat{M}^T}$ is not given by any “natural” topology. Due to the isomorphisms in~\eqref{eq:group-isomorphisms} this section deals with finding an ordering, that is similar to the tensor product case by introducing a basis for $\patSet{\mat{M}}$. 
	%The following section deals with bases for the pattern $\patSet{\mat{M}}$. 
	Multiplying all the formulae in the following section with $\mat{M}$ leads to the same properties for the generating group $\gGroup{\mat{M}}$, hence starting with $\patSet{\mat{M}^T}$ also a basis for $\gGroup{\mat{M}^T}$ can be constructed.

	For a fixed regular matrix $\mat{M}$ we define the set of vectors
	\begin{equation}\label{eq:basisPM}
		\vect{y}_j :=
		%		 f^{-1}(\frac{1}{\epsilon_{d-d_{\mat{M}}+j}}\vect{e}_{d-d_{\mat{M}}+j}) =
		\mat{R}^{-1}\frac{1}{\epsilon_{d-d_{\mat{M}}+j}}\vect{e}_{d-d_{\mat{M}}+j}, \quad j=1,\ldots,d_{\mat{M}}
		\text{,}
	\end{equation}
	where $\mat{R}$ denotes the left basis transform in the Smith normal form~\eqref{eq:SNF} and $\vect{e}_j$ denotes the $j$-th unit vector. These vectors are linear independent, because $\mat{R}$ has full rank. Using $\mat{R}$ as a change of basis, we see from~\citep[Lemma 2.4]{LangemannPrestin2010} that $\lattice{\mat{M}} = \lattice{\mat{E}\mat{R}}$ and
	\begin{equation*}
	\lattice{\mat{M}}=\lattice{\mat{E}\mat{R}}=\{\vect{y}\,|\,\mat{E}\mat{R}\vect{y}\in\mathbb Z^{d}\}\cong\lattice{\mat{E}}=\{\vect{x}\,|\,\mat{E}\vect{x}\in\mathbb Z^{d}\}\text{, }
	\end{equation*}
	where $f: \lattice{\mat{E}\mat{R}} \to \lattice{\mat{E}}, f(\vect{y}) = \mat{R}\vect{y}$ is an isomorphism between $\lattice{\mat{M}}$ and $\lattice{\mat{E}}$.

	The scaled unit vectors $\epsilon_{d-d_{\mat{M}}+j}^{-1}\vect{e}_{d-d_{\mat{M}}+j},\quad j=1,\ldots,d_{\mat{M}}$, form a basis for $\patSet{\mat{E}}$, i.e. there is a unique representation for each $\vect{x}\in\patSet{\mat{E}}$ of the form
	\begin{equation}\label{eq:linear-independence-expl}
		\vect{x}  = \calcMod{%\left(%
		\sum_{j=1}^{d_{\mat{M}}} \lambda_j\tfrac{1}{\epsilon_{d-d_{\mat{M}}+j}}\vect{e}_{d-d_{\mat{M}}+j}%\right)
		}{\patSet{\mat{E}}},\quad \text{ where } 0 \leq \lambda_j < \epsilon_{d-d_{\mat{M}}+j}%, j=1,\ldots,d_{\mat{M}}
		\text{.}
	\end{equation}
	The other unit vectors are not scaled due to $j\leq d-d_{\mat{M}} \Leftrightarrow \epsilon_j = 1$, and hence vanish in such a summation with respect to the congruence classes, i.e. $\bmod \mat{I}$. By using the inverse of the isomorphism $f^{-1}: \lattice{\mat{E}} \to \lattice{\mat{E}\mat{R}}, f^{-1}(\vect{y}) = \mat{R}^{-1}\vect{y}$, the vectors $\{\vect{y}_1,\ldots,\vect{y}_{d_{\mat{M}}}\}$ form a basis of $\patSet{\mat{M}}$.

	\begin{rem}
		The lattice generated by $\{\vect{y}_1,\ldots,\vect{y}_{d_{\mat{M}}}\}$ is called rank-$d_{\mat{M}}$-lattice.  Further the ordering of the basis elements is crucial, because each vector $\vect{y}_j$ has to span a cycle of length $\epsilon_{d-d_{\mat{M}}+j}$ (with respect to $\calcMod{\cdot}{\patSet{\mat{M}}}$) in this notation. The value $d_{\mat{M}}$ is also called the dimension of the pattern $\patSet{\mat{M}}$.
	\end{rem}

	Using the basis we obtain an ordering of the elements of $\patSet{\mat{M}}$ by using the lexicographical ordering of
	\begin{equation*}
		\patSet{\mat{M}} = \left(
		\calcMod{\sum_{j=1}^{d_{\mat{M}}} \lambda_j\vect{y}_{j}}{\patSet{\mat{M}}}
		%\bmod \mat{I}
		\right)_{(\lambda_1,\ldots,\lambda_{d_{\mat{M}}})=\vect{0}}^{\epsilon_{d-d_{\mat{M}}+1}-1,\ldots,\epsilon_d-1}
	\end{equation*}
	on the set $\mathbb E_{\mat{M}} = \{0,1,\ldots,\epsilon_{d-d_{\mat{M}}+1}\}\times\cdots\times\{0,1,\ldots,\epsilon_d\} $ of indices.

	\begin{ex}\label{ex:lattice2}
		For the matrix $\mat{M} = \bigl(\begin{smallmatrix}4&-3\\4&5\end{smallmatrix}\bigr)$ from Ex.~\ref{ex:lattice1} we see, that $d_{\mat{M}} = 1$ and the only basis vector is $\vect{y}_1 = \mat{R}^{-1}\vect{e}_2 = \tfrac{1}{32}\begin{pmatrix} 12&1\end{pmatrix}^T = \bigl(\tfrac{3}{8},\tfrac{1}{32}\bigr)^T$, which is emphasised in Fig.~\ref{fig:Pattern-Example} (left). Hence by writing $\patSet{\mat{M}} = \{ \calcMod{k\vect{y}_1}{\patSet{\mat{M}}}, k=0,1,\ldots,31\}$ we even get an ordinary ordering of the elements.
	\end{ex}
	Applying these ideas on the generating group, leads to a basis of $\gGroup{\mat{M}^T}$ denoted by
	\begin{equation}\label{eq:basisGM}
		\vect{h}_j := \mat{R}^T\vect{e}_{d-d_{\mat{M}}+j},\quad j=1,\ldots,d_{\mat{M}}\text{.}
	\end{equation}
	This leads to the biorthogonality of these two bases, which is shown in the following
	\begin{lem}\label{lem:basis-relation}
		Let $\mat{M}\in\mathbb Z^{d\times d}$ be regular. Then the bases of $\patSet{\mat{M}}$ and $\gGroup{\mat{M}^T}$ given in~\eqref{eq:basisPM} and~\eqref{eq:basisGM} are biorthogonal, more precisely
		\begin{equation}
			\forall i,j\in\{1,\ldots,d_{\mat{M}}\}:\quad
			\langle \vect{h}_{j},\vect{y}_{i}\rangle
			=
			\begin{cases}
				\frac{1}{\epsilon_{d-d_{\mat{M}}+i}}&\mbox{ if } i=j\text{,}\\
				0&\mbox{ else.}
			\end{cases}
		\end{equation}
	\end{lem}
	\begin{proof}
		For arbitrary $i,j\in\{1,\ldots,d_{\mat{M}}\}$ it holds
		\begin{eqnarray*}
			\begin{split}
				\langle \vect{h}_j,\vect{y}_i\rangle
			&= \langle \mat{R}^T\vect{e}_{d-d_{\mat{M}}+j}, \mat{R}^{-1}\frac{1}{\epsilon_{d-d_{\mat{M}}+i}}\vect{e}_{d-d_{\mat{M}}+i} \rangle
	%		= \frac{1}{\epsilon_{d-d_{\mat{M}}+j}}\langle \vect{e}_{d-d_{\mat{M}}+i}
	%		, \mat{R}\mat{R}^{-1}\vect{e}_{d-d_{\mat{M}}+i} \rangle
	%		\\
	%		&
			= \frac{1}{\epsilon_{d-d_{\mat{M}}+i}}\vect{e}_{d-d_{\mat{M}}+j}^T\vect{e}_{d-d_{\mat{M}}+i}\text{.}
			\end{split}
		\end{eqnarray*}
	\end{proof}
	% section bases_for_the_pattern_and_the_generating_group (end)
	% end input bases
	%
	%
	%
	%
	%	\input{fastfourier}
	\section{The fast Fourier transform on $\patSet{\mat{M}}$} % (fold)
	\label{sec:the_fast_fourier_transform_on_matm}
	This section is devoted to derive a fast Fourier transform on $\patSet{\mat{M}}$ using its characterisation obtained in the previous section. The idea generalises an algorithm described in~\citep{MeBrGu83,MeSp81}. It does not depend on finding a prime factor decomposition of $\mat{M}$ as mentioned in~\citep[Eq. (35)]{MeSp81}. It uses the Smith normal form similar to~\citep{MeBrGu83}, but our approach is able to omit the rearrangement steps.

	A more general approach for abelian groups and their characters can be found in~\citep{Auslander1995}. In contrast, our approach uses properties of the pattern and generating group to deduce a fast algorithm working on arrays.

	\subsection{Properties of the multivariate Fourier transform} % (fold)
	\label{sub:properties_of_the_multivariate_fourier_transform}
	Based on~\citep[Lemma 2.1]{LangemannPrestin2010}, there exist permutation matrices $\mat{P}_{\vect{h}}, \mat{P}_{\vect{y}}$ on $\gGroup{\mat{M}^T}$ and $\patSet{\mat{M}}$ respectively, such that
	\begin{equation}\label{eq:FourierKronecker}
		\mathcal F(\mat{M}) = \mat{P}_{\vect{h}}
		\left(\mathcal F_{\epsilon_1}\otimes \mathcal F_{\epsilon_2}\otimes\cdots\otimes\mathcal F_{\epsilon_d}\right)
		\mat{P}_{\vect{y}}\text{, }
	\end{equation}
	where 
	\begin{equation*}
		\mathcal F_{\epsilon} = \frac{1}{\sqrt{\epsilon}}\left(\euler^{-2 \pi \imag h\epsilon^{-1}g}\right)_{h,g=0}^{\epsilon-1},\quad \epsilon \in \mathbb N^+
	\end{equation*}
	denote the elementary divisors of the Smith normal form of $\mat{M}$. Using the bases constructed in Section~\ref{sec:bases_for_the_pattern_and_the_generating_group}, this section develops properties which simplify~\eqref{eq:FourierKronecker} to
	\begin{equation}\label{eq:FourierKronecker2}
		\mathcal F(\mat{M}) =
		\mathcal F_{\epsilon_{d-d_{\mat{M}}+1}}\otimes \mathcal F_{\epsilon_{d-d_{\mat{M}}+2}}
		\otimes\cdots\otimes \mathcal F_{\epsilon_d}\text{,}
	\end{equation}
	where the first factors $\mathcal F_{\epsilon_1},\ldots,\mathcal F_{\epsilon_{d-d_{\mat{M}}}}$ of the Kronecker product in~\eqref{eq:FourierKronecker} can be omitted due to $\mathcal F_{1} = \euler^0 = 1$. The following theorem characterises conditions for the ordering which are necessary for $\mathcal F(\mat{M})$ to fulfill~\eqref{eq:FourierKronecker2}. We start with the orderings  using the bases from~\eqref{eq:basisPM} and~\eqref{eq:basisGM} of $\patSet{\mat{M}}$ and $\gGroup{\mat{M}^T}$ from Section~\ref{sec:bases_for_the_pattern_and_the_generating_group} and look at their lexicographical orderings, i.e.
	\begin{equation}\label{eq:orderings}
		\begin{split}
		\patSet{\mat{M}} &
		= \left(\calcMod{\sum_{j=1}^{d_{\mat{M}}} \lambda_j\vect{y}_j}{\patSet{\mat{M}}}
		\right)_{%(\lambda_1,\ldots,\lambda_{d_{\mat{M}}})
		\vect{\lambda}
		\in\mathbb E_{\mat{M}}}
	%	\\
		&\hspace{-1em}\text{ and }\hspace{1em}&
	%	&\\
		\gGroup{\mat{M}^T} &
		= \left(\calcMod{\sum_{j=1}^{d_{\mat{M}}} \mu_j\vect{h}_j}{\gGroup{\mat{M}^T}}
		\right)_{%(\mu_1,\ldots,\mu_{d_{\mat{M}}})
		\vect{\mu}\in\mathbb E_{\mat{M}}}\text{,}
		\end{split}
	\end{equation}
	where $\vect{\lambda} = (\lambda_1,\ldots,\lambda_{d_{\mat{M}}})$ and $\vect{\mu}=(\mu_1,\ldots,\mu_{d_{\mat{M}}})$.
	\begin{thm}[A basis for the Kronecker product]\label{thm:direcctKroneckerBasis}\ \\
	The orderings from~\eqref{eq:orderings} used for the construction of $\mathcal F(\mat{M})$ fulfill~\eqref{eq:FourierKronecker2}.
	% if and only if
	%	\begin{equation}\label{eq:basis-relation}
	%		\forall j \in \{d-d_{\mat{M}}+1,\ldots,d\} : 
	%		\left\langle\vect{h}_{\beta_{\mathcal G}^{-1}(j)},\vect{y}_{\beta_{\mathcal P}^{-1}(j)}\right\rangle
	%		= 
	%		\left(\vect{h}_{\beta_{\mathcal G}^{-1}(j)}\right)^T\vect{y}_{\beta_{\mathcal P}^{-1}(j)}
	%		\equiv \frac{1}{\epsilon_j} \bmod 1 \text{.}
	%	\end{equation}
	\end{thm}
	\begin{proof}
	%	Let~\eqref{eq:basis-relation} be satisfied.
		Given any $j \in \{d-d_{\mat{M}}+1,\ldots,d\}$, in particular $\epsilon_j > 1$, orderings in~\eqref{eq:orderings}
		and Lemma~\ref{lem:basis-relation}
		%and~\eqref{eq:basis-relation}
		are used to prove validity of~\eqref{eq:FourierKronecker2} by induction over $d_{\mat{M}}$. For $d_{\mat{M}}=1$ the matrix $\mathcal F(\mat{M})$ is identical to the one-dimensional case, hence Lemma~\ref{lem:basis-relation} implies~\eqref{eq:FourierKronecker2} %and vice verse
		. 

		For $d_{\mat{M}} > 1$ let $\lambda_1,\ldots,\lambda_{d_{\mat{M}}-1}$ and $\mu_1,\ldots,\mu_{d_{\mat{M}}-1}$ be given, each fulfilling $0\leq \lambda_j,\mu_j < \epsilon_{d-d_{\mat{M}}+j},\quad j=1,\ldots,d_{\mat{M}}$, and denote	
		\begin{equation*}
			c = \frac{1}{\sqrt{\epsilon_{d-d_{\mat{M}}+1}\cdot\ldots\cdot\epsilon_{d-1}}}\exp\left(- 2 \pi \imag \sum\limits_{j=1}^{d_{\mat{M}}-1}\lambda_j\epsilon_{d-d_{\mat{M}}+j}^{-1}\mu_j\right)\text{.}
		\end{equation*}
		Then the submatrix of $\mathcal F(\mat{M})$ induced by the elements
		\begin{align*}
	%		\begin{split}
			\vect{x}_l  = \calcMod{\sum_{j=1}^{d_{\mat{M}}} \lambda_j\vect{y}_j}%
			{\patSet{\mat{M}}}
			% \bmod \mat{I}
			,\quad l=\lambda_{d_{\mat{M}}} = 0,\ldots,\epsilon_d-1\\
			\intertext{and}
			\vect{z}_k = \calcMod{\sum_{j=1}^{d_{\mat{M}}} \mu_j\vect{h}_j}%
			{\gGroup{\mat{M}^T}}% \bmod \mat{M}^T
			,\quad k=\mu_{d_{\mat{M}}} = 0,\ldots,\epsilon_d-1
	%		\end{split}
		\end{align*}
		can be written as
		\begin{equation}\label{eq:FormSubmatrix}
			\begin{split}
			\frac{1}{\sqrt{m}}\left(\euler^{-2 \pi \imag \vect{z}_k^T\vect{x}_l} \right)_{k,l=0}^{\epsilon_d-1}%\\
			&= \frac{1}{\sqrt{m}}\left(
			%\exp\left(%
			\euler^{-2 \pi \imag \left(\sum\limits_{j=1}^{d_{\mat{M}}} \mu_j\vect{h}_j\right)^T\left(\sum\limits_{j=1}^{d_{\mat{M}}} \lambda_j\vect{y}_j\right)
			%\right)
			}
			 \right)_{k,l=0}^{\epsilon_d-1}\\
			&= \frac{c}{\sqrt{\epsilon_d}}
			\left(\euler^{-2 \pi \imag k\vect{h}^T_jl\vect{y}_j}\right)_{k,l=0}^{\epsilon_d-1}
			%\stackrel{\text{Lem.~\ref{lem:basis-relation}}}{
			=%} 
			c \mathcal F_{\epsilon_d}.
		\end{split}
		\end{equation}
		Hence, this submatrix of $\mathcal F(\mat{M})$, generated by arbitrary elements with fixed
		\\%HACK
		$\lambda_1, \ldots, \lambda_{d_{\mat{M}}-1}, \mu_1, \ldots, \mu_{d_{\mat{M}}-1}$ is $\mathcal F_{\epsilon_{d}}$. This is the last factor of the Kronecker product in~\eqref{eq:FourierKronecker2}, where $c$ represents one element of the complete previous product. This previous product up to the last but one factor is already fulfilling the form of~\eqref{eq:FourierKronecker2} by the  induction hypothesis. So Lemma~\ref{lem:basis-relation} implies~\eqref{eq:FourierKronecker2}.
	\end{proof}
	\begin{rem}\label{rem:basis-generalization}
		For a slightly loosened version of Lemma~\ref{lem:basis-relation}, i.e. that $\langle \vect{h}_i,\vect{y}_i\rangle \equiv \epsilon_{d-d_{\mat{M}}+i}^{-1}\bmod1$, Theorem~\ref{thm:direcctKroneckerBasis} holds for any pair of biorthogonal bases for $\patSet{\mat{M}}$ and $\gGroup{\mat{M}^T}$ and even the reverse implication is true: If $\mathcal F(\mat{M})$ is of the form~\eqref{eq:FourierKronecker2}, then there exist two bases fulfilling the slightly loosened version of Lemma~\ref{lem:basis-relation} and provide an ordering for the matrix columns and rows as denoted in~\eqref{eq:orderings}. The proof is just the reverse steps of the proof of Theorem~\ref{thm:direcctKroneckerBasis}. Hence a pair of biorthogonal bases for $\patSet{\mat{M}}$ and $\gGroup{\mat{M}}$ fulfilling the modified Lemma~\ref{lem:basis-relation} is necessary and sufficient for the Fourier matrix to fulfill~\eqref{eq:FourierKronecker2}.
	\end{rem}
	% subsection properties_of_the_multivariate_fourier_transform (end)
	% section the_fast_fourier_transform_on_mat{m_ (end)

	\subsection{Fast Fourier transform} % (fold)
	\label{sub:fast_fourier_transform}
	Following the approach of Mersereau et. al.~\citep{MeBrGu83} or in a more general matter also described by Auslander et al.~\citep{AuJoJo96}, we can now apply fast Fourier transformation techniques by adapting the multivariate Cooley-Tukey-Algorithm. In addition to the latter general approach we present a complete algorithm using any biorthogonal bases for the pattern $\patSet{\mat{M}}$ and its dual group $\gGroup{\mat{M}^T}$ from Lemma~\ref{lem:basis-relation}, cf. also Remark \ref{rem:basis-generalization}, and analyse the complexity of the algorithm. We also avoid the reindexing mentioned by the former authors using the presented basis and their coefficient vectors to arrange the vectors in the Fourier transform.

	Given a basis $\{\vect{y}_1,\ldots,\vect{y}_{d_{\mat{M}}}\}$ of $\patSet{\mat{M}}$, we can decompose every $\vect{y}\in\patSet{\mat{M}}$ u\-nique\-ly, i.e.
	\begin{equation*}
	%\forall \vect{y}\in\patSet{\mat{M}}
	\exists!\,\vect{\lambda}\in\mathbb E_{\mat{M}}\,:\, \vect{y} = \calcMod{\sum_{k=1}^{d_{\mat{M}}}\lambda_k\vect{y}_k}{\patSet{\mat{M}}}\text{.}% \bmod \mat{I}	
	\end{equation*}
	Using this decomposition every vector $\vect{b} = \left(b_{\vect{y}}\right)_{\vect{y}\in\patSet{\mat{M}}}$ can also be addressed with $\vect{b} = \left(b_{\lambda}\right)_{\lambda \in \mathbb E_{\mat{M}}}$. Let $\mathcal G := \mathcal F_{\epsilon_{d-d_{\mat{M}}+1}}\otimes\mathcal F_{\epsilon_{d-d_{\mat{M}}+2}}\otimes\cdots\otimes\mathcal F_{\epsilon_{d-1}}\in\mathbb C^{n\times n},\quad n=\tfrac{m}{\epsilon_d}$ denote the Fourier transform with respect to the basis vectors $\vect{y}_1,\ldots,\vect{y}_{d_{\mat{M}}-1}$. Using the indexing with respect to the position inside the cycles, the Fourier transform reads
	\begin{equation}
		\begin{split}
		\vect{\hat b} &= \left(\mathcal G \otimes \mathcal F_{\epsilon_d}\right)\vect{b}\\
		&= 
		\begin{pmatrix}
			&&&\\ %empty row for spacing
			\mathcal G_{1,1}\mathcal F_{\epsilon_d} & \mathcal G_{1,2}\mathcal F_{\epsilon_d} & \cdots & \mathcal G_{1,n}\mathcal F_{\epsilon_d}\\
			&&&\\ %empty row for spacing
			&&&\\ %empty row for spacing
			\mathcal G_{2,1}\mathcal F_{\epsilon_d} & \mathcal G_{2,2}\mathcal F_{\epsilon_d} & \cdots & \mathcal G_{2,n}\mathcal F_{\epsilon_d}\\
			&&&\\ %empty row for spacing
			\vdots			&				& 		& 	\vdots \\
			\mathcal G_{n,1}\mathcal F_{\epsilon_d} & \mathcal G_{n,2}\mathcal F_{\epsilon_d} & \cdots & \mathcal G_{n,n}\mathcal F_{\epsilon_d}\\
			&&&\  %empty row for spacing
		\end{pmatrix}
		\begin{pmatrix}
			b_{(0,0,\ldots,0)} \\ \vdots \\ b_{(0,0,\ldots,\epsilon_d-1)} \\
			b_{(0,\ldots,0,1,0)} \\ \vdots \\ b_{(0,\ldots,0,1,\epsilon_d-1)} \\ b_{(0,\ldots,0,2,0)}\\\vdots\\b_{(\epsilon_1-1,\epsilon_2-1,\ldots,\epsilon_d-1)}
		\end{pmatrix}\text{.}
		\end{split}
		\label{eq:KroneckerFFT}
	\end{equation}
	This enables us to split the computation into two parts:
	\\%HACK
	First calculating $\mathcal F_{\epsilon_d}
	\left(b_{(\lambda_1,\ldots,\lambda_{d_{\mat{M}}})}\right)_{\lambda_{d_{\mat{M}}}=0}^{\epsilon_d-1}$ for fixed values of $\lambda_1,\ldots,\lambda_{d_{\mat{M}}-1}$. Following this blockwise Fourier transform, an interleaved addressing due to fixed $\epsilon_{d_{\mat{M}}}$ is used to compute the multiplications with $\mathcal G$. This leads to an implementation denoted in Algorithm~\ref{alg:fourier}.

	\begin{lstlisting}[label=alg:fourier, caption={A fast Fourier transform on $\patSet{\mat{M}}$ in \emph{Mathematica} notation, where \lstinline!Fourier! denotes any implementation of the one-dimensional Fourier transform.}, float,texcl, language=Mathematica, floatplacement=!tb]
		FourierOnPattern[$\epsilon$_, b_] := Block[{hatb},
		% $\epsilon = (\epsilon_i)_{i=1}^{d_{\mat{M}}}$: vector containig the elementary divisors of $\mat{M}$
		% b: vector of input values given as  $\vect{b} = \left(b_{\lambda}\right)_{\lambda \in \mathbb E_{\mat{M}}}$
		%
		% hatb: The Fourier transform $\hat{\vect{b}} = \mathcal F(M)\vect{b}$
			If [Length[$\epsilon$] == 1, Return[Fourier[b]];
			% Perform the transforms on blocks of size $\epsilon_d$
			Do [
				hatb[[{$\mu_1,\ldots,\mu_{d_{\mat{M}-1}}$},All]] = Fourier[b[[{$\mu_1,\ldots,\mu_{d_{\mat{M}-1}}$},All]]];
				,   {$\mu_1$,1,$\epsilon_1$},$\ldots$,{$\mu_{d_{\mat{M}}}$,1,$\epsilon_{d_{\mat{M}}-1}$}
				];
			% Perform transform on the first $d_{\mat{M}-1}$ cycles
			% recursively on interleaved blocks
			Do [
				% sec denotes $d_{\mat{M}}-1$ times the term All
				hatb[[sec,$\xi$]] = FourierOnPattern[{$\epsilon_1,\ldots,\epsilon_{d_{\mat{M}-1}}$},hatb[[sec,$\xi$]]];
				,   {$\xi$,1,$\epsilon_{d_{\mat{M}}}$}
				];
		Return[hatb];
		];
	\end{lstlisting}
	\begin{thm}[computational complexity of the FFT on $\mathbb T^d$]\label{thm:FFTonTd}\ \\%Hack
	Let $\mat{M}\in\mathbb Z^{d\times d},\ m=|\det\mat{M}|>0$ be given. The Fourier transform $\vect{\hat b} = \mathcal F(\mat{M})\vect{b}$ using~\eqref{eq:KroneckerFFT} can be computed with $O(m\log m)$ operations.
	\end{thm}
	\begin{proof}
		Let any implementation of the one-dimensional FFT be given, i.e. with computational costs $c_{\text{FFT}}k\log k + O(k)$ for an input of $k$ coefficients, where e.g. $c_{\text{FFT}} = \frac{34}{9}$ as shown in~\citep{JoFr06}. Then the proof is again by induction over $d_{\mat{M}}$:
		The first step is a blockwise computation of $\mathcal F_{\epsilon_d}
		\left(b_{(\lambda_1,\ldots,\lambda_{d_{\mat{M}}})}\right)_{\lambda_{d_{\mat{M}}}=0}^{\epsilon_d-1}$ for fixed values of $\lambda_1,\ldots,\lambda_{d_{\mat{M}}-1}$. These are $n := \epsilon_1\cdot\ldots\cdot\epsilon_{d-1}$ Fourier transforms of computation cost $c_{\text{FFT}}\epsilon_d\log\epsilon_d + O(\epsilon_d)$ each. After that, for each fixed $0\leq \lambda_{d_{\mat{M}}} < \epsilon_d$  we have to compute the Fourier transform using the Matrix $\mathcal G$, which are $\epsilon_d$ transforms, where each of them needs $c_{\text{FFT}}n\log n + O(n)$ time by induction hypothesis. This leads to 
		\begin{equation}\label{eq:fft-complexity}
			\begin{split}
				&n\left(c_{\text{FFT}}\epsilon_d\log\epsilon_d + O(\epsilon_d)\right)
				+ \epsilon_d\left(c_{\text{FFT}}n\log n + O(n)\right)\\
				=\ &
				c_{\text{FFT}}n\epsilon_d\log\epsilon_d + O(m)
				+ c_{\text{FFT}}n\epsilon_d\log n + O(m)\\
				=\ &
				c_{\text{FFT}}m\left( \log\epsilon_d + \log n\right) + O(m)
				< 
				c_{\text{FFT}}m\log m + O(m)\text{.}
				\end{split}			
		\end{equation}
	\end{proof}
	The proof also reveals the optimisation possibilities of the multivariate Fourier transform in comparison to the one-dimensional case with the same size $m$ of points: The first step consists of $n$ Fourier transforms of the same size $\epsilon_d$, where each transform acts on a distinct (blockwise) subset of the given data. This means one can use a vectorial SIMD implementation~\citep{Franchetti2002} and compute this term in a parallel implementation in $O(\epsilon_d\log\epsilon_d)$. The same holds for the $\epsilon_d$ Fourier transforms in the second step, which act on distinct (interleaved) subsets on the immediate result, which is in $O(n\log n)$, using again a vectorial parallel implementation.
	% subsection fast_fourier_transform (end)
	%
	%
	%
	%
	%
	% was: \input{fastwavelet}
	\section{Fast wavelet transform} % (fold)
	\label{sec:fast_wavelet_transform}
	The same techniques presented in Section~\ref{sec:the_fast_fourier_transform_on_matm} can also be applied to the corresponding periodic wavelet transform. We first introduce subspaces of $L^2(\mathbb T^d)$ that are constructed using translates of one function and decompose these into $j \geq 2$ subspaces. We present some properties from~\citep{LangemannPrestin2010} concerning these spaces and derive a fast decomposition. We also generalise some properties from the one-dimensional case in~\citep{Se98}, e.g. the scaling property, that is extended by directional information.

	For convenience, we leave out the shift $\calcMod{\cdot}{X}$ back into the set of congruence class representants in the summations on $X=\patSet{\mat{M}}$
	 and $X=\gGroup{\mat{M}^T}$, whenever it is clear from context.

	\subsection{Translation invariant spaces}
	A subspace $L \subset L^2(\mathbb T^d)$ is called %\emph{
	%invariant to the translations of $\patSet{\mat{M}}$ %}
	%or %\emph{
	$\mat{M}$-invariant %}
	%for short
	if
	\begin{equation*}
		\forall \vect{y}\in\patSet{\mat{M}}: f \in L \Rightarrow T(\vect{y})f \in L\text{.}
	\end{equation*}
	Such a subspace can easily be constructed given a function $f \in L^2(\mathbb T^d)$. We define
	\begin{equation*}
		V_{\mat{M}}^{f} = \text{span}\{T(\vect{y})f\,:\,\vect{y} \in \lattice{\mat{M}}\} = \text{span}\{T(\vect{y})f\,:\,\vect{y} \in \patSet{\mat{M}}\}
	\end{equation*}
	as the span of all translates of $f$ (with respect to a matrix $\mat{M}$).

	Then for two functions $f,g\in L^2(\mathbb T^d)$ it holds~\citep[Theorem 3.3]{LangemannPrestin2010} that $g\in V_{\mat{M}}^f$ if and only if there exists a vector $\vect{a}  = (a_{\vect{y}})_{\vect{y}\in\patSet{\mat{M}}}$ with its Fourier Transform $\vect{\hat a} = (\hat a _{\vect{h}})_{\vect{h}\in\gGroup{\mat{M}^T}}\\ = \mathcal F(\mat{M})\vect{a}$ such that
	\begin{equation}\label{eq:subspaces-coeffs}
		\forall\,\vect{h} \in \gGroup{\mat{M}^T}\ \forall\,\vect{k}\in\mathbb Z^d : c_{\vect{h}+\mat{M}^T\vect{k}}(g) = \hat a_{\vect{h}}c_{\vect{h}+\mat{M}^T\vect{k}}(f)\text{.}
	\end{equation}
	Then 
	\begin{equation*}
		g = \sum_{\vect{y}\in\patSet{\mat{M}}} a_{\vect{y}}T(\vect{y})f\text{.}
	\end{equation*}
	%For details and a proof, see \cite{LangemannPrestin2010}. 
	If we further look at a decomposition $\mat{M} = \mat{J}\mat{N}$ we see from the definition of $T(\vect{y})$, that $V_{\mat{N}}^g \subset V_{\mat{M}}^f$ and it further holds that
	\begin{equation*}
		\overline{\mathcal F ( \mat{N} ) }\left(T(\vect{y})g\right)_{\vect{y}\in\patSet{\mat{N}}}
		= \sqrt{\frac{n}{m}}\mat{A}\overline{\mathcal F (\mat{M})}\left(T(\vect{y})f\right)_{\vect{y}\in \patSet{\mat{M}}}\text{,}
	\end{equation*}
	where for a certain order of $\gGroup{\mat{M}^T}$ it follows
	\begin{equation}\label{eq:OrderingA}
		\mat{A} = \left(\text{diag}\left(\hat a_{\vect{k}+\mat{J}^T\vect{l}}\right)_{\vect{h}\in\gGroup{\mat{N}^T}}\right)_{\vect{l}\in\gGroup{\mat{J}^T}} \in \mathbb C^{n\times m}\text{.}
	\end{equation}
	Applying this to a set of functions $g_1,\ldots,g_{|\det \mat{J}|}$, whose translates are pairwise orthogonal, we get a decomposition of $V_{\mat{M}}^f$ into subspaces, cf.~\citep[Theorem 4.1]{LangemannPrestin2010}.
	\subsection{Properties of a multivariate decomposition}
	Let $\mat{J},\mat{N}\in\mathbb Z^{d\times d}$ be regular matrices and $\mat{M} = \mat{J}\mat{N}$. We denote the bases of the corresponding patterns by $\{\vect{x}_1,\ldots,\vect{x}_{d_\mat{M}}\}, 
	\{\vect{y}_1,\ldots,\vect{y}_{d_\mat{N}}\}$ and $ \{\vect{z}_1,\ldots,\vect{z}_{d_\mat{J}}\}$ for $\patSet{\mat{M}},\ 
	\patSet{\mat{N}}$ and $\patSet{\mat{J}}$.
	These bases are ordered with respect to the cycle lengths, cf. \eqref{eq:basisPM}, e.g. $\vect{y}_j$ corresponds to a cycle of length $\epsilon_{d-d_{\mat{M}}+j}$, i.e. $k\vect{y}_j \equiv 0 \bmod \mat{I} \Leftrightarrow k=h\epsilon_{d-d_{\mat{M}}+j},\quad h\in\mathbb Z$. We can characterise subspace $\patSet{\mat{N}}$ of $\patSet{\mat{M}}$ using the following Lemma.

	\begin{lem}[projection]\label{lem:BasisprojektioninPM}
		Let $\mat{J},\mat{N}\in\mathbb Z^{d\times d}$ be regular matrices and $\mat{M} = \mat{J}\mat{N}$. There exists a matrix  $\mat{P} \in \mathbb N_0^{d_{\mat{M}}\times d_{\mat{N}}}$ such that 
		for an arbitrary $\vect{w}\in\patSet{\mat{N}}$ we have
		\begin{equation}\label{eq:adressing}
			\exists \vect{\mu} \in\mathbb E_{\mat{N}}\ \ 
			\exists \vect{\lambda} \in\mathbb E_{\mat{M}}\ \ :
			\vect{w} = 
			\sum_{k=1}^{d_{\mat{N}}}\mu_k\vect{y}_k
			= \sum_{l=1}^{d_{\mat{M}}}\lambda_l\vect{x}_l\text{,}
		\end{equation}
		where $\vect{\lambda} = \mat{P}\vect{\mu}$ holds.
	\end{lem}
	\begin{proof}
		Equation~\eqref{eq:adressing} follows from the fact that $\patSet{\mat{N}} \subset \patSet{\mat{M}}$ and decomposing $\vect{w}$ in both bases. Because every basis vector $\vect{y}_k\in\patSet{\mat{M}}$ can also be decomposed, we have
		\begin{equation*}
			\exists \vect{p}_k = (p_{l,k})_{l=1}^{d_{\mat{M}}}\in\mathbb N^{d_{\mat{M}}}_0
			: \vect{y}_k = \sum_{l=1}^{d_{\mat{M}}}p_{l,k}\vect{x}_l,\quad k=1,\ldots,d_{\mat{N}}\text{.}
		\end{equation*}
		This is then applied to
		\begin{equation*}
			\vect{w} = \sum_{k=1}^{d_{\mat{N}}}\mu_k\vect{y}_k
			= \sum_{k=1}^{d_{\mat{N}}}\mu_k \sum_{l=1}^{d_{\mat{M}}}p_{l,k}\vect{x}_l
			= \sum_{l=1}^{d_{\mat{M}}}\sum_{k=1}^{d_{\mat{N}}}\mu_k p_{l,k}\vect{x}_l
			= \sum_{l=1}^{d_{\mat{M}}}\lambda_l\vect{x}_l\text{.}
		\end{equation*}
	\end{proof}
	Using this property we can state a generalisation of~\citep[Theorem 4.1.2a]{Se98}, which deals with the translation invariance of subspaces in the one-dimensional case. For any regular regular matrix $\mat{J}$ fulfilling $\mat{N} = \mat{J}^{-1}\mat{M}\in\mathbb Z^{d\times d}$, the $\mat{M}$-invariant space is $\mat{N}$-invariant. Additionally, we are also able to characterise dimensions and directions of the subspace $\patSet{\mat{M}}$ depending on $\mat{J}$ as described in the following theorem. In the one-dimensional case this would just be the value $\mat{J}\in\mathbb N$ which is the number of points added per point in $\patSet{\mat{N}}, \mat{N}\in\mathbb N$. Usually one would choose this factor to be 2 to get the classical dyadic decomposition scheme of one-dimensional wavelet analysis. In the multivariate case however, different matrices $\mat{J}$ of the same modulus of its determinant exist, which enable different extensions of a pattern $\patSet{\mat{N}}$. These are first stated for some special matrices $\mat{J}$, i.e. $d_{\mat{J}}=1$, in the following theorem, which includes all matrices of absolute determinant 2. We will discuss the general case after the theorem.

	\begin{thm}[multivariate scaling property]\label{thm:Skalierungseigenschaft}\ \\
		Let $\mat{J},\mat{N} \in\mathbb{Z}^{d\times d}$ be regular matrices and $\mat{M}=\mat{J}\mat{N}$, such that the dimension $d_{\mat{J}} = 1$. Denote by $\epsilon_j^{\mat{M}},\epsilon_j^{\mat{J}},\epsilon_j^{\mat{N}}$, $j=1,\ldots,d$, their elementary divisors. Then it holds
		\begin{enumerate}[1)]
			\item for $\mat{N}^{-1}\vect{z}_1 \not \in \text{span}\{\vect{y}_1,\ldots,\vect{y}_{d_{\mat{N}}}\}$, that
		\begin{enumerate}[a)]
			\item $d_{\mat{M}} = d_{\mat{N}} + 1$
			\item $\exists\vect{x}_l \in \{\vect{x}_1,\ldots,\vect{x}_{d_{\mat{M}}}\}$
			\begin{equation*}
				\mat{N}^{-1}\vect{z}_1 = \lambda\vect{x}_l \bmod \mat{I},\quad\lambda\in\{1,\ldots,\epsilon_d^{\mat{J}}-1\}\text{ and } \epsilon_l^{\mat{M}} = \epsilon_{d}^{\mat{J}}\text{,}
			\end{equation*}
		\end{enumerate}

		\item for $\mat{N}^{-1}\vect{z}_1 \in \text{span}\{\vect{y}_1,\ldots,\vect{y}_{d_{\mat{N}}}\}$, that
		\begin{enumerate}[a)]
			\item $d_{\mat{M}} = d_{\mat{N}}$
			\item the decompositions
			\begin{equation*}
				\begin{split}
					\mat{N}^{-1}\vect{z}_1 = \sum_{l=1}^{d_{\mat{M}}}\lambda_l\vect{x}_l,\quad  \vect{\lambda}\in\mathbb E_{\mat{M}}\text{ and }
					\mat{N}^{-1}\vect{z}_1 = \sum_{k=1}^{d_{\mat{N}}}\mu_k\vect{y}_k, \quad \vect{\mu} \in \mathbb Q^d
				\end{split}
			\end{equation*}
			\noindent exist and fulfill
			\begin{equation}\label{eq:scalingBetweenPatterns}
				\vect{\lambda} = \frac{1}{\epsilon_d^{\mat{J}}}\mat{P}\vect{\mu}\text{, }
			\end{equation}
		\end{enumerate}
		where $\mat{P} \in \mathbb N_0^{d_{\mat{M}}\times d_{\mat{M}}}$ is regular.
		\end{enumerate}
	\end{thm}
	\begin{proof}
		It holds $\mat{J}\vect{z}_1 = \mat{M}\mat{N}^{-1}\vect{z}_1 \in \mathbb Z^d$, hence $\mat{N}^{-1}\vect{z}_1\in\patSet{\mat{M}}$. If $\,\mat{N}^{-1}\vect{z}_1 \not \in \text{span}\{\vect{y}_1,\ldots,\vect{y}_{d_{\mat{N}}}\}$, then $\text{span}\{\vect{y}_1,\ldots,\vect{y}_{d_{\mat{N}}}\} \subset \text{span}\{\vect{y}_1,\ldots,\vect{y}_{d_{\mat{N}}},\mat{N}^{-1}\vect{z}_1\}$, which also holds for the spans if we restrict the weighted sums to $\mathbb E_{\mat{N}}$ and $\mathbb E_{\mat{N}}\times \mathbb E_{\mat{J}}$ respectively.

		From~\eqref{eq:decomposition} we have for any $\vect{x}\in \patSet{\mat{M}}$, the unique decomposition into two elements
		\begin{equation*}
	%		\exists ! \vect{y}\in\patSet{\mat{N}}\ 
	%		\exists ! \vect{z}\in \patSet{\mat{J}}\,:\,
			\vect{x} = \calcMod{\vect{y} + \mat{N}^{-1}\vect{z}}{\patSet{\mat{M}}},\quad \vect{y}\in\patSet{\mat{N}}$ and $\vect{z}\in\patSet{\mat{J}}\text{.}
		\end{equation*}
		This can be decomposed uniquely using the bases of $\patSet{\mat{N}}$ and $\patSet{\mat{J}}$, which reads
		\begin{equation*}
			\exists ! \vect{\mu} \in \patIndex{\mat{N}}\ 
		    \exists ! \theta \in \patIndex{\mat{J}}\,:\,
			\vect{x} = \sum_{k=1}^{d_{\mat{N}}}\mu_k\vect{y}_k + \theta\mat{N}^{-1}\vect{z}_1\text{, }
		\end{equation*}
		where the second term is only one summand due to $d_{\mat{J}}=1$. Looking at assumption 1 we conclude, that the set $\{\vect{y}_1,\ldots,\vect{y}_{d_{\mat{N}}},\mat{N}^{-1}\vect{z}_1\}$ is linear independent, which is a). Statement b) follows from the fact that $\patSet{\mat{N}}$ is a subgroup of $\patSet{\mat{M}}$ and hence all cycles of the first are also existent in the latter one. This also implies that exactly one basis vector $\vect{x}_l$ spans the new cycle $\mat{N}^{-1}\vect{z}_1$.

		Using the same approach, it holds for the second case that $\mat{N}^{-1}\vect{z}_1\in\operatorname{span}\{\vect{y}_1,\ldots,\vect{y}_{d_{\mat{N}}}\}$ and hence we can decompose %(keeping in mind that $\patSet{\mat{J}}$ lying on a grid, so does $\mat{N}^{-1}\patSet{\mat{J}}$):
		\begin{equation*}
			\exists ! \vect{\nu} \in \patIndex{\mat{N}}\ 
		    \exists ! \vect{\mu} \in \mathbb Q^{d_{\mat{N}}}\ 
		    \exists ! \theta \in \patIndex{\mat{J}}\,:\,	
			\mat{N}^{-1}\vect{z}_1 = \sum_{k=1}^{d_{\mat{N}}}\nu_k\vect{y}_k + \theta\sum_{k=1}^{d_{\mat{N}}}\mu_k\vect{y}_k,\\
		\end{equation*}
		which is a). Using $\epsilon_d^{\mat{J}}\vect{z}_1 \in \mathbb Z^d$ and the fact that
		\begin{equation*}
	\exists! \vect{\mu} \in \patIndex{\mat{N}} : \epsilon_d^{\mat{J}}\mat{N}^{-1}\vect{z}_1 = \sum_{k=1}^{d_{\mat{N}}} \mu_k\vect{y}_k
		\end{equation*}
		we can apply Lemma~\ref{lem:BasisprojektioninPM} and see that
		\begin{equation*}
	\exists! \vect{\lambda}\in\patIndex{\mat{M}}
			: \mat{N}^{-1}\vect{z}_1
			= \sum_{l=0}^{d_{\mat{M}}} \lambda_l\vect{x}_l
			= \frac{1}{\epsilon_d^{\mat{J}}}\sum_{k=0}^{d_{\mat{N}}} \mu_k\vect{y}_k
			=
			\frac{1}{\epsilon_d^{\mat{J}}}
			\sum_{k=0}^{d_{\mat{N}}} \mu_k
			 \sum_{l=1}^{d_{\mat{M}}}p_{l,k}\vect{x}_l\text{, }
		\end{equation*}
		which is~\eqref{eq:scalingBetweenPatterns}.
	\end{proof}
	The equality~\eqref{eq:scalingBetweenPatterns} also generalises the notation, that the one-dimensional Fourier transform, when  the sampling points double in their number (halfening the distance), the frequencies double. These one-dimensional Fourier transforms inside the multivariate one can be seen in Theorem~\ref{thm:direcctKroneckerBasis} and are characterised in Theorem~\ref{thm:Skalierungseigenschaft} above with their directions. In fact even for the first case a cycle, i.e. a one-dimensional Fourier transform, gets increased in size, that was $\epsilon_j^{\mat{N}} = 1$ before and hence was not part of the basis.

	Theorem~\ref{thm:Skalierungseigenschaft} can also be generalised to more than one cycle in $\mat{J}$ by decomposing the basis of $\patSet{\mat{J}}$ into two distinct parts $B_1,B_2$ of basis vectors fulfilling the first and the second case of the Theorem. Then the dimension gets $d_{\mat{M}} = d_{\mat{N}} + |B_1|$ and for each basis vector in $B_2$ a scaling property as in~\eqref{eq:scalingBetweenPatterns} holds.

	\subsection{A fast decomposition algorithm} % (fold)
	\label{sub:the_decomposition_algorithm}
	Given a decomposition $\mat{M}=\mat{J}\mat{N}$ of a regular integral matrix $\mat{M}$ and the functions\\%Hack
	$f,g_1,\ldots,g_{|\det\mat{J}|}$ such that 
	\begin{equation*}
	V_{\mat{M}}^f = \bigoplus_{j=1}^{|\det\mat{J}|} V_{\mat{N}}^{g_j} \text{,}\quad\text{hence } g_j = \sum_{\vect{y}\in\patSet{\mat{M}}}b_{j,\vect{y}}T(\vect{y})f\in V_{\mat{M}}^f,\quad j=1,\ldots,|\det\mat{J}|	
	\end{equation*}
	 we can decompose any $\gamma = \sum_{\vect{y}\in\patSet{\mat{M}}} a_{\vect{y}}T(\vect{y})f\in V_{\mat{M}}^f$ using any pair of bases
	$\{\vect{h}_1,\ldots,\vect{h}_{d_{\mat{M}}}\}$ of $\gGroup{\mat{M}^T}$ and 
	$\{\vect{k}_1,\ldots,\vect{k}_{d_{\mat{N}}}\}$ of $\gGroup{\mat{N}^T}$ using the following steps:
	% Let further denote $\{\vect{l}_1,\ldots,\vect{l}_{d_{\mat{J}}}\}$ a basis of $\gGroup{\mat{J}^T}$.
	\begin{enumerate}
		\item Compute $\vect{\hat a} = \mathcal F(\mat{M})\vect{a}$ and $\vect{\hat b}_{l} = \mathcal F(\mat{M})\vect{b}_l,\quad l=1,\ldots,|\det \mat{J}|$
		\item Calculate $\mat{P}$ for the two bases $\{\vect{h}_1,\ldots,\vect{h}_{d_{\mat{M}}}\}$ of $\gGroup{\mat{M}^T}$ and 
		$\{\vect{k}_1,\ldots,\vect{k}_{d_{\mat{N}}}\}$ of $\gGroup{\mat{N}^T}$ using Lemma~\ref{lem:BasisprojektioninPM}.
		\item Calculate for each basis vector of a basis $\{\vect{l}_1,\ldots,\vect{l}_{d_{\mat{J}}}\}$ of $\gGroup{\mat{J}^T}$ the decomposition
		\begin{equation*}
			\mat{N}^T\vect{l}_j =
			\calcMod{\sum_{i=1}^{d_{\mat{M}}} q_{i,j}\vect{h}_i}{\gGroup{\mat{M}^T}},\quad (q_{i,j})_{i=1}^{d_{\mat{M}}}\in\mathbb E_{\mat{M}}, \quad j=1,\ldots,d_{\mat{J}}
		\end{equation*}
		and apply these to address $\mat{N}^T\vect{l}\in\gGroup{\mat{M}^T},\vect{l}\in\gGroup{\mat{J}^T}$ using their coefficient vector $\vect{\lambda}_{\vect{l}}\in\mathbb E_{\mat{M}}$.
		\item Compute $\vect{\lambda}=\calcMod{\mat{P}\vect{\mu}}{\mathbb E_{\mat{M}}}$ for each $\vect{h}\in\gGroup{\mat{N}^T}$, which can be reached running through all $\vect{\mu}\in\mathbb E_{\mat{N}}$ and compute
		\begin{equation*}
			\hat d_{j,\vect{h}} = \frac{1}{\sqrt{|\det\mat{J}|}}\sum_{\vect{l}\in\gGroup{\mat{J}^T}} \hat b_{j,\vect{h}+\mat{N}^T\vect{l}} \hat a_{\vect{h}+\mat{N}^T\vect{l}},\quad j=1,\ldots,|\det{\mat{J}}|,\quad\vect{h} \in \gGroup{\mat{N}^T}
		\end{equation*}
		addressing $\vect{\hat a}, \vect{\hat b}$ using $\calcMod{\vect{\lambda}+\vect{\lambda}_{\vect{l}}}{\mathbb E_{\mat{M}}}$ and $\vect{\hat d}_{j,\vect{h}}$ using $\vect{\mu}$.
		\item Perform the inverse Fourier Transform with each $\vect{\hat d}_{\!j}$ to get the resulting decomposition
		\begin{equation*}
			\gamma = \sum_{j=1}^{|\det{\mat{J}}|}\sum_{\vect{y}\in\patSet{\mat{N}}} d_{j,\vect{y}}T(\vect{y})g_j\text{.}
		\end{equation*}
	\end{enumerate}
		For the special case $|\det{\mat{J}}|=2$, where $g_1$ represents in some sense the low frequent part of $f$ and $g_2$ the high frequent part of $f$, this describes one step in a dyadic fast wavelet transform. For any following step in the decomposition, i.e. $\mat{N} = \mat{J'}\mat{O}$ to split $V_{\mat{N}}^{g_1}$ into two spaces, we can apply the steps 2--4 to $\vect{\hat b}_1$.% (the Fourier transform of $d_{1,\vect{y}}$ is already given by the first decomposition).

		 Given the number $m = |\det\mat{M}|$ of data points  as input we obtain the complexity of computation for the algorithm. We further note the influence of the dimension $d$, though it is constant with respect to the amount of sampled data points $m$. The steps 2 and 3 have to solve at most $|\det{\mat{J}}|+d$  linear systems of equations with at most $d$ unknowns each, the first step is by~\eqref{eq:fft-complexity} $2c_{\text{FFT}}m\log m + O(m)$, the same without the factor $2$ holds for the last step, due to $|\det \mat{J}|n \log n < m\log m$. Finally, step 4 needs $c|\det\mat{J}|m = O(m)$ steps, where the constant is just depending on the speed of multiplication on that machine. In total this is
	\begin{equation*}
		3c_{\text{FFT}}m\log m + c|\det\mat{J}|m +c_1|\det\mat{J}|d^4 + c_2(m +d^4)
	\end{equation*}
	where $c,c_1,c_2$ are constants just depending on speed of multiplication (on a specific machine) and they are especially independent of $d,m$ and $|\det\mat{J}|$. Performing multiple decompositions on the same set of data would only affect the second and third term, because the Fourier transform would be computed once in the beginning and for every wavelet level in total once at the end, which are in total $m$ coefficients.
	% subsection the_decomposition_algorithm (end)
	% section fast_wavelet_transform (end)
	% end of input fastwavelet
	%
	%
	%
	%
	% was: \input{example}
	\section{Example}
	%{\color{red} Some description of the test and the result}
	\label{sec:example}

	\subsection{Fourier transform}
	As an example for the Fourier transform we look at different patterns $\patSet{\mat{M}}$ with the same number $m=|\patSet{\mat{M}}|=|\det{\mat{M}}|$ of points. The pattern normal form~\citep[Lemma 2.5]{LangemannPrestin2010} can be used to look at different classes of matrices that induce the same pattern. Restricting the example to the case $d=2$, all normal forms have the form 
	\begin{equation}\label{eq:FFTexampleMatrix}
	\mat{M} = \begin{pmatrix} l & i\\0& k\end{pmatrix}\text{, where }
	m=kl,\quad k,l\in\mathbb N^+\text{  and }0\leq i<k\text{.}
	\end{equation}
	Furthermore all different cycle lengths $\epsilon_1,\epsilon_2$, which may occur, are given by all divisors $i$ of $k$.

	Given any implementation of the one-dimensional Fourier transform, the Algorithm~\ref{alg:fourier} can be easily implemented. For any of the recursive calls in line 16, the set of data a subroutine is working on, is distinct from all the other. Hence, the calls can also be parallelized.

	We compare the implementation of a serial and a parallel algorithm of the multivariate Fourier transform, where the first one is compared to the usual one-dimensional Fourier transform on the same amount of data, i.e. an $m$-dimensional vector. The algorithms are all using the implementation of \lstinline!Fourier! given by \emph{Mathematica} 8 running on an Intel Pentium 4 Core 2 Quad with $2\,$Ghz each and $4\,$GB memory.
		%
		%
		% was: 	\input{table-example-timing}
		\begin{table}[htbp]
			\centering
			\begin{tabular}{rrrrrr}\toprule $i$ & cycles & \mbox{serial (sec.)} & factor & \mbox{parallel (sec.)} & gain\\\midrule
				$1$ & $\begin{pmatrix}
				 4194304
		\end{pmatrix}$ & $0.266664$ & $1.02242$ 

				 & $0.270647$ & $0.98528$\\

				$2$ & $\begin{pmatrix}
				 2 & 
				 2097152
		\end{pmatrix}$ & $0.471722$ & $1.80864$ 

				 & $0.343969$ & $1.37141$\\

				$4$ & $\begin{pmatrix}
				 4 & 
				 1048576
		\end{pmatrix}$ & $0.468220$ & $1.79522$ 

				 & $0.299426$ & $1.56373$\\

				$8$ & $\begin{pmatrix}
				 8 & 
				 524288
		\end{pmatrix}$ & $0.457126$ & $1.75268$ 

				 & $0.282309$ & $1.61924$\\

				$16$ & $\begin{pmatrix}
				 16 & 
				 262144
		\end{pmatrix}$ & $0.457321$ & $1.75342$ 

				 & $0.286432$ & $1.59661$\\

				$32$ & $\begin{pmatrix}
				 32 & 
				 131072
		\end{pmatrix}$ & $0.469080$ & $1.79851$ 

				 & $0.299551$ & $1.56595$\\

				$64$ & $\begin{pmatrix}
				 64 & 
				 65536
		\end{pmatrix}$ & $0.466902$ & $1.79016$ 

				 & $0.315000$ & $1.48223$\\

				$128$ & $\begin{pmatrix}
				 128 & 
				 32768
		\end{pmatrix}$ & $0.572192$ & $2.19386$ 

				 & $0.412570$ & $1.38690$\\

				$256$ & $\begin{pmatrix}
				 256 & 
				 16384
		\end{pmatrix}$ & $0.920226$ & $3.52826$ 

				 & $0.741036$ & $1.24181$\\

				$512$ & $\begin{pmatrix}
				 512 & 
				 8192
		\end{pmatrix}$ & $1.144126$ & $4.38672$ 

				 & $0.530716$ & $2.15582$\\

				$1024$ & $\begin{pmatrix}
				 1024 & 
				 4096
		\end{pmatrix}$ & $0.949403$ & $3.64013$ 

				 & $0.424802$ & $2.23493$\\

				$0$ & $\begin{pmatrix}
				 2048 & 
				 2048
		\end{pmatrix}$ & $0.907286$ & $3.47865$ 

				 & $0.411705$ & $2.20373$\\
		\bottomrule
			\end{tabular}
			\caption{For all possible cycles of $\mat{M}$ from \eqref{eq:FFTexampleMatrix}, where $k=l=2^{11}$ we compare a serial and a parallel implementation of the two-dimensional FFT. The first is compared to the one-dimensional FFT of $m=|\det\mat{M}|=2^{22}$ points of data using \lstinline!Fourier! in \emph{Mathematica}. This took $0.260816$ seconds, which leads to the factor (col. 4). The table lists all possible elementary divisors $i$ of $k$, which corresponds to the obtained cycles. The gain represents the factor, the parallel implementation gains---on 4 cores---compared to the serial one. All times are obtained taking the mean value of $50$ measurements using \lstinline!AbsoluteTiming!.}\label{tab:2048}
		\end{table}
		% end of table
		%
		%
		
	Choosing $k=l=2048$ and a randomly generated set of $m=2^{22}$ data values, the computational times for all divisors $i$ of $k$ are depicted in Table~\ref{tab:2048}. The factor denoted in the column after the serial one is obtained by comparison to the one-dimensional Fourier transform. The multidimensional data is obtained by sampling along the cycles of $\patSet{\mat{M}}$, which represent the sampling directions. This data is then used to perform the lattice Fourier transform and measuring its serial and parallel computation time. The last column denotes the parallel gain.

	While the serial implementation is equal to the one-dimensional case for $i=1$, it is slower for the cases, where the Fourier transform has to be applied to rows and columns. The parallel implementation is able to gain a factor of $1.2$ to $2.8$ on $4$ cores.

	The lattice Fourier transform reduces with a change of basis to the usual multivariate Fourier transform. Further investigations of the parallelization of the multivariate Fourier transform can be found e.g. based on the FFTW in~\citep{Pippig:2012}.

	\subsection{Wavelet transform}
	To extend the example from the previous subsection, we look at the Dirichlet type wavelets for different decompositions $\mat{M} = \mat{J}\mat{N}$, where
	\begin{equation*}%\label{eq:matrixset}
		\mat{J} \in \left\{
		\begin{pmatrix}	2&0\\0&1 \end{pmatrix},
		\begin{pmatrix}	1&0\\0&2 \end{pmatrix},
		\begin{pmatrix}	1&-1\\1&1 \end{pmatrix}	
		\right\} = \{\mat{J}_x,\mat{J}_y,\mat{J}_d\}=\mathcal J\text{.}
	\end{equation*}
	Let $\patSet{\mat{M}} = \lattice{\mat{M}}\cap[-\tfrac{1}{2},\tfrac{1}{2})^d$ and $\gGroup{\mat{M}} = \mat{M}\patSet{\mat{M}}$ be full sets of congruence class representants of $\lattice{\mat{M}}$ and $\mathbb Z^d$. We denote by $r(\vect{k}) = r_{\mat{M}}(\vect{k}), \vect{k}\in\mathbb Z^d$ the number of superficial hyperplanes of the parallelepiped $\mat{M}^T[-\tfrac{1}{2},\tfrac{1}{2})^d$ a point $\vect{h} = \calcMod{\vect{k}}{\gGroup{\mat{M}^T}}\in\gGroup{\mat{M}^T}$ is lying on, i.e.
	\begin{equation*}
		r(\vect{k}) = \#\left\{
		j : |\vect{h}^T\mat{M}^{-1}|_j = \frac{1}{2},\quad\vect{h} = \calcMod{\vect{k}}{\gGroup{\mat{M}^T}}
		\right\}\text{.}
	\end{equation*}
	The Dirichlet kernel $\varphi_{\mat{M}}: \mathbb T^d\to\mathbb R$ is given by its Fourier coefficients
	\begin{equation*}
		c_{\vect{k}}(\varphi_{\mat{M}}) = 
		\begin{cases}
			\frac{1}{\sqrt{m}}2^{-\frac{r(\mat{k})}{2}}&\mbox{ if } \mat{M}^{-T}\vect{k}\in[-\tfrac{1}{2},\tfrac{1}{2}]^d\text{,}\\
			0 & \mbox{ else,}
		\end{cases}\quad\vect{k}\in\mathbb Z^d\text{.}
	\end{equation*}
	The translates $T(\vect{y})\varphi_{\mat{M}},\ \vect{y}\in\patSet{\mat{M}}$ are linear independent and orthonormal~\citep[Theorem 6.4]{LangemannPrestin2010} and it holds $\varphi_{\mat{N}} \in V_{\mat{M}}^{\varphi_{\mat{M}}}$ due to 
	\begin{equation*}
		\begin{split}
			c_{\vect{k}}(\varphi_{\mat{N}}) &= \hat a_{\vect{h}}
			c_{\vect{k}}(\varphi_{\mat{M}}),\quad\vect{k}\in\mathbb Z^d,\quad\vect{h} = \calcMod{\vect{k}}{\gGroup{\mat{M}^T}}		
			\text{,}\quad\text{ where }\\
			\hat a_{\vect{h}} &= \begin{cases}
				2^{\frac{1}{2}(1+r_{\mat{M}}(\vect{h}) - r_{\mat{N}}(\vect{h}))}&\mbox{ if } \mat{N}^{-T}\vect{k} \in \left[-\tfrac{1}{2},\tfrac{1}{2}\right]^d\cap\patSet{\mat{M}^T}\text{, }\\
				0&\mbox{ else, }
			\end{cases}
		\end{split}
	\end{equation*}
	see~\eqref{eq:subspaces-coeffs} and~\citep[Lemma 6.5]{LangemannPrestin2010}. The corresponding wavelet $\psi_{\mat{N}}$ fulfilling $V_{\mat{M}}^{\varphi_{\mat{M}}} = V_{\mat{N}}^{\varphi_{\mat{N}}}\oplus V_{\mat{N}}^{\psi_{\mat{N}}}$ is given by
	\begin{equation*}
		c_{\vect{k}}(\psi_{\mat{N}}) = \begin{cases}
			c_{\vect{k}}(\varphi_{\mat{M}})\hat a_{\calcMod{\vect{k}+\mat{N} T\vect{g}}{\gGroup{\mat{M}^T}}}\euler^{-2\pi \imag \vect{k}^T\mat{N}^{-1}\vect{y}}&\mbox{ if } \mat{M}^{-T}\vect{k}\in[-\tfrac{1}{2},\tfrac{1}{2}]^d\text{,}\\
			0&\mbox{ else, }
		\end{cases}
	\end{equation*}
	where $\vect{y}\in\patSet{\mat{J}}\backslash\{\vect{0}\}$ and $\vect{g}\in\gGroup{\mat{J}^T}\backslash\{\vect{0}\}$ are uniquely determined due to $|\det{\mat{J}}| = 2$ for each $\mat{J}\in\mathcal J$.

	As test functions we choose two centered box splines $B_{\Xi}, B_{\Psi}: \mathbb R^2\to \mathbb R$ with $\Xi = \begin{pmatrix}
		\pi & 0 & \tfrac{\pi}{8}\\
		0 & \pi & \tfrac{\pi}{8} 
	\end{pmatrix}$, 	$\Psi = \begin{pmatrix}
			\pi & 0 & \tfrac{\pi}{8} & 0 & \tfrac{\pi}{8}\\
			0 & \pi & 0 & \tfrac{\pi}{8} & \tfrac{\pi}{8} 
		\end{pmatrix}$, see~\citep{boxSplines} and Fig.~\ref{fig:BoxSplines}. The function $B_{\Xi}$ is a piecewise linear polynomial, $B_{\Psi}$ a piecewise cubic polynomial.
	\begin{figure}
		\centering
		\includegraphics[width=.47\textwidth]{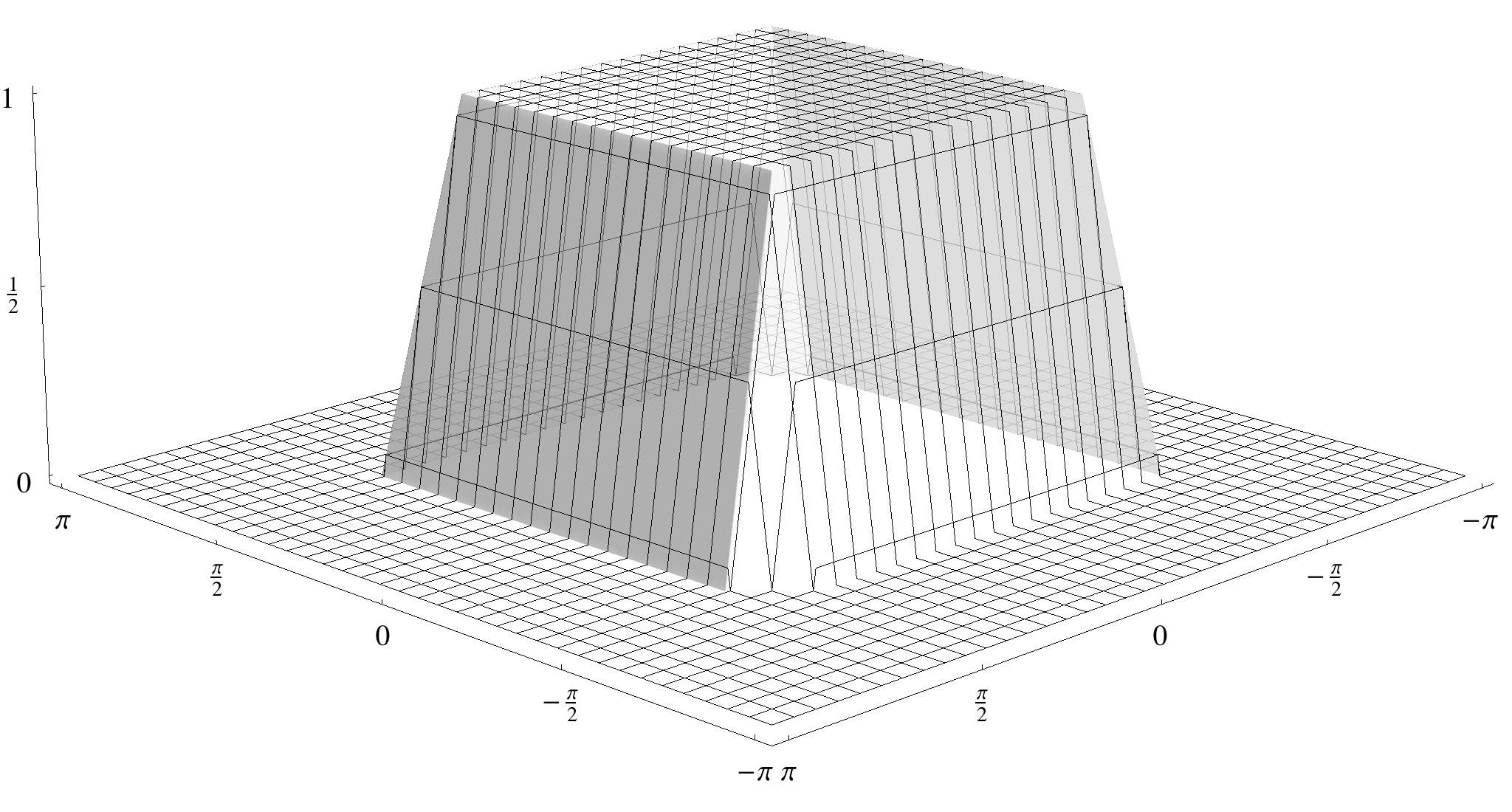}\hspace{.03\textwidth}
		\includegraphics[width=.47\textwidth]{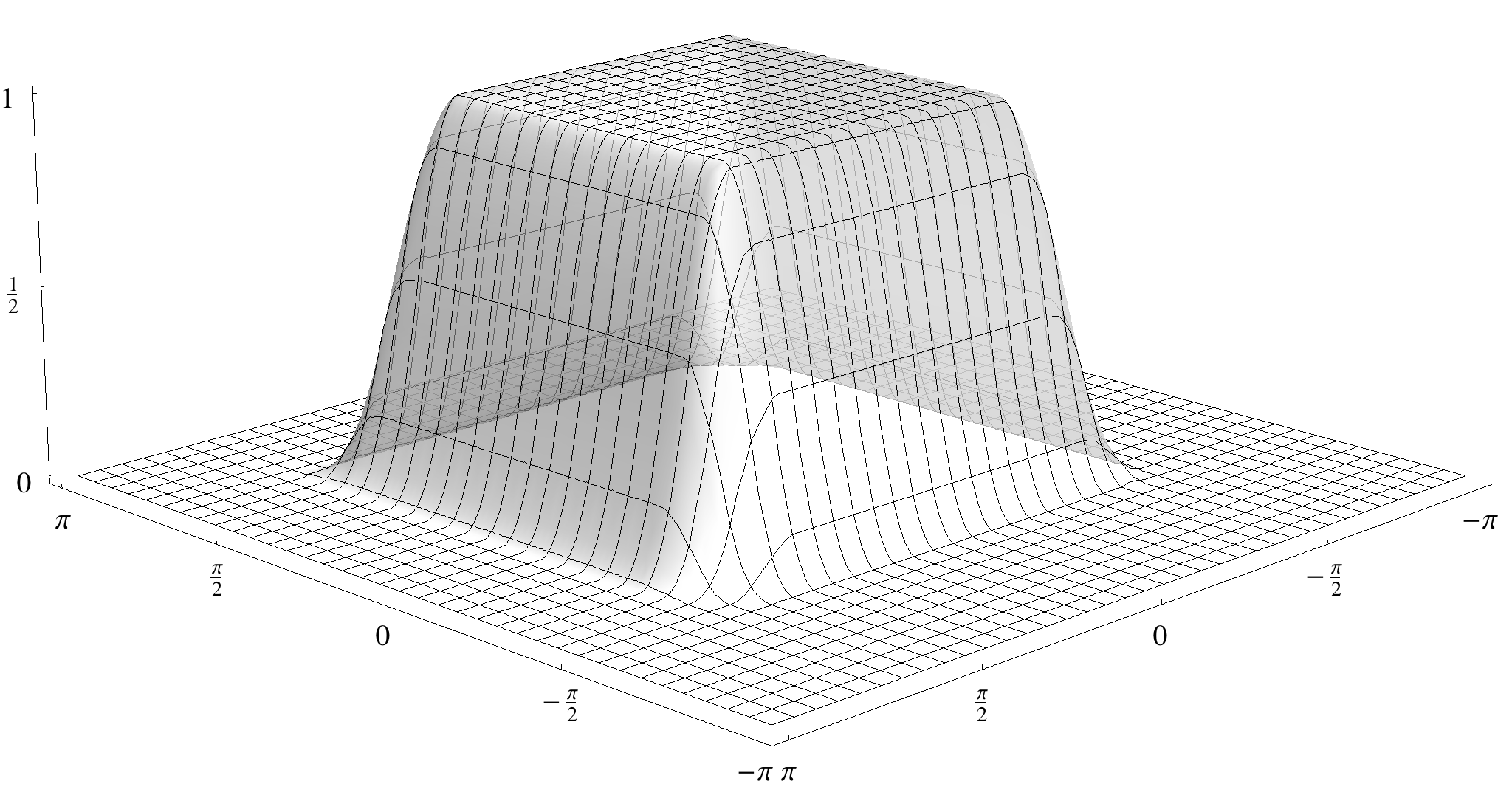}
		\caption{The centered box splines $B_{\Xi}$ (left) and $B_{\Psi}$ (right) used as test functions for the Dirichlet type wavelet decomposition.}
		\label{fig:BoxSplines}
	\end{figure}
		We choose $\mat{M} = \begin{pmatrix}512&0\\0&512\end{pmatrix}$ and
		 the function $\tilde{f_{\Xi}}\in V_{\mat{M}}^{\varphi_{\mat{M}}}$ %(see Fig.~\ref{subfig:tfXi})
		 is obtained by sampling $B_{\Xi}$ at the points $2\pi\vect{y},\ \vect{y}\in\patSet{\mat{M}}$. From these samples we perform a change of basis from the interpolatory or Lagrangian basis into the basis of orthonormal translates $\{T(\vect{y})\varphi_{\mat{M}}\}_{\vect{y}\in\patSet{\mat{M}}}$ by implementing the well known change of basis, which is for this case stated e.g. in~\citep[Corollary 3.7]{LangemannPrestin2010}.

		We perform the first step of the wavelet transform, i.e. we decompose $\tilde f = \tilde f_V+\tilde f_W$, where $\tilde f_V \in V_{\mat{N}}^{\varphi_{\mat{N}}}$ and $\tilde f_W \in V_{\mat{N}}^{\psi_{\mat{N}}}$. The function $\tilde f_W$ is in the $17$th wavelet space of a dyadic pyramid of wavelet spaces due to $|\det \mat{N}| = 2^{17}$.
		\begin{figure}
	%		\subfigure[sampled Image {$\tilde f_\Xi$}]{
	%		\includegraphics[width=.23\textwidth]{B1-Img.png}
	%		\label{subfig:tfXi}
	%		}
			\subfigure[{$|\tilde f_W|$} using {$\mat{J}_x$}]{
			\includegraphics[width=.315\textwidth]{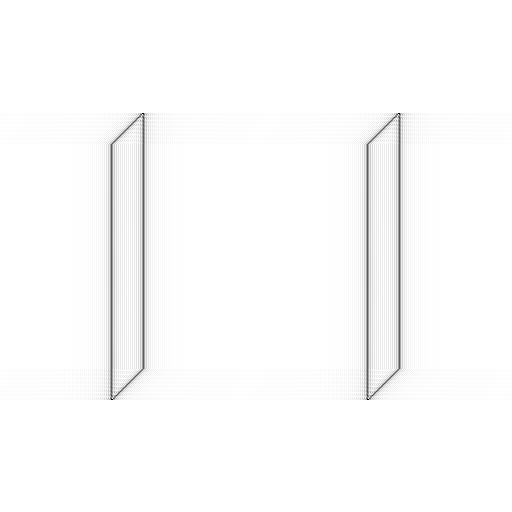}
			\label{subfig:B1X}
			}
			\subfigure[{$|\tilde f_W|$} using {$\mat{J}_y$}]{
			\includegraphics[width=.315\textwidth]{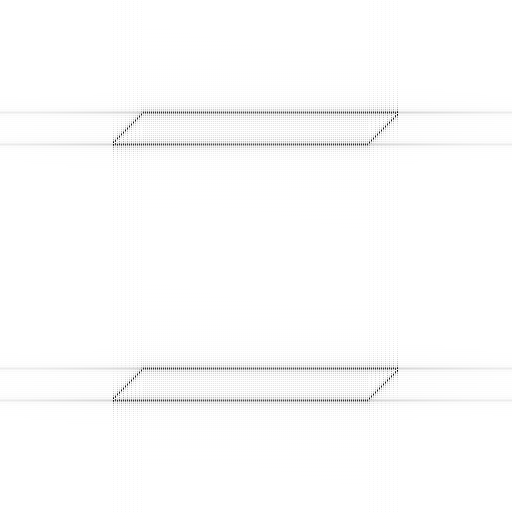}
			\label{subfig:B1Y}
			}
			\subfigure[{$|\tilde f_W|$} using {$\mat{J}_d$}]{
			\includegraphics[width=.315\textwidth]{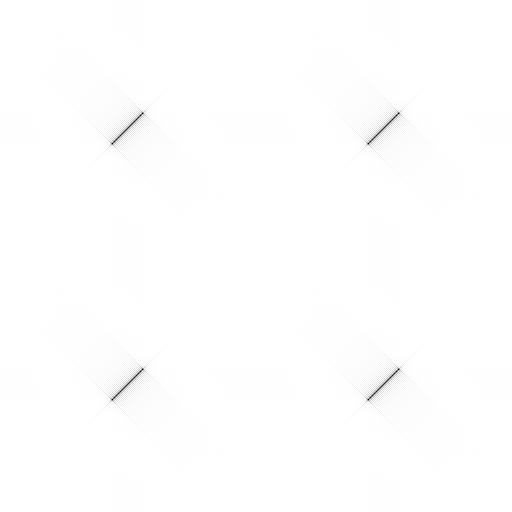}
			\label{subfig:B1D}
			}
			\caption{The three possible functions $|\tilde f_W|$ as wavelet spaces of $B_{\Xi}$ drawn as a contour plot, where white indicates zero and black the highest value. %, i.e. (a) 1, (b) $4.27576*10^{-5}$ (c) $3.51441*10^{-5}$ (d)
			}
			\label{fig:B1-Decomp}
		\end{figure}

		Depending on the matrix $\mat{J}$ we obtain different directional information in the wavelet space for the box spline $B_{\Xi}$: The lines of discontinuity of the first derivative are divided into three sets: The discontinuities along the diagonal can be obtained using $\mat{J}_d$ (see Fig.~\ref{subfig:B1D}), while discontinuities parallel to the axes are depicted in the wavelet spaces generated by $\mat{J}_x$ and $\mat{J}_y$ (see Figs.~\ref{subfig:B1X} and~\ref{subfig:B1Y}). The functions $\tilde f_W$ are shown in their absolute value to emphasise the nonzero parts of these wavelet spaces. The oscillations seen as grey values between the lines or to the rim are due to the Dirichlet type wavelets. The same holds for the diagonal discontinuities that are also visible in the first two decompositions.
	\begin{figure}
	%	\subfigure[sampled Image {$\tilde f_\Psi$}]{
	%	\includegraphics[width=.23\textwidth]{B2-Img.png}
	%	}
		\subfigure[{$|\tilde f_W|$} using {$\mat{J}_x$}]{
		\includegraphics[width=.315\textwidth]{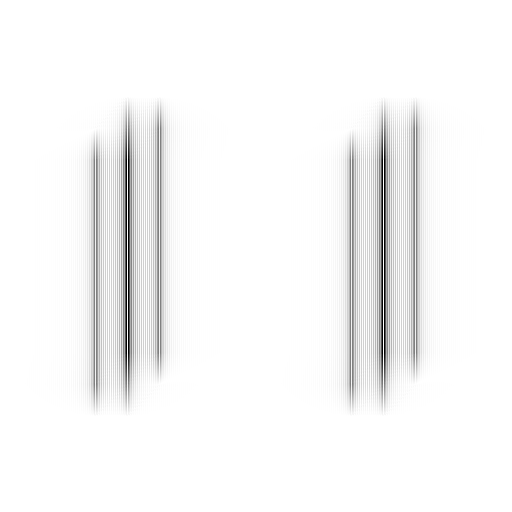}
		}
		\subfigure[{$|\tilde f_W|$} using {$\mat{J}_y$}]{
		\includegraphics[width=.315\textwidth]{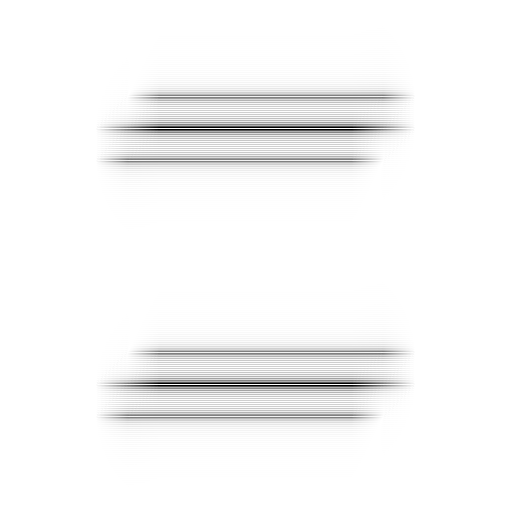}
		}
		\subfigure[{$|\tilde f_W|$} using {$\mat{J}_d$}]{
		\includegraphics[width=.315\textwidth]{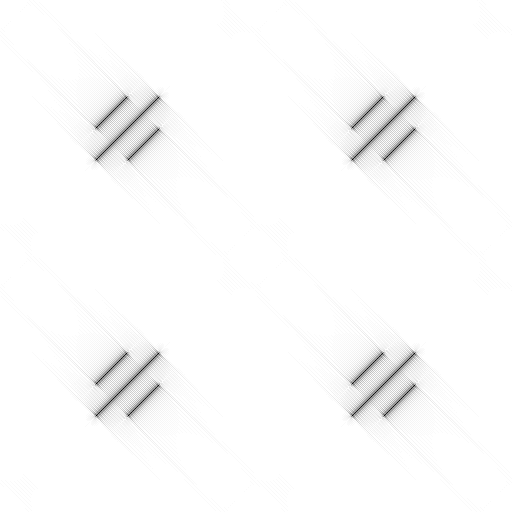}
		}
		\caption{The three possible functions $|\tilde f_W|$ as wavelet spaces of $B_{\Psi}$ drawn as a contour plot, where white indicates zero and black the highest value. %, i.e.: (a) 1 (b) $1.03269*10^{-7}$ (c) $2.36402*10^{-7}$ (d) $1.23354*10^{-9}$.	
		}
		\label{fig:B2-Decomp}
	\end{figure}

	For the second example we apply the Dirichlet type wavelets to $B_{\Psi}$ (see Fig.~\ref{fig:B2-Decomp}). The wavelet spaces contain the discontinuities of the third derivative%, which in the first case only exist in a distributional sense
	. Here, we see the same oscillations as in the previous case, but the decomposition with respect to certain directions is more clear. The union of all three wavelet spaces represents the set of lines, on which the third derivative is discontinuous along the directional derivative orthogonal to the corresponding line.
	% end of example
	%
	%
\bibliographystyle{model1b-num-names}
\bibliography{references}
\end{document}